\documentclass[reqno]{amsart}
\usepackage{mathrsfs}
\usepackage{color}
\usepackage{amsmath}
\usepackage{amsfonts}
\usepackage{amssymb}
\usepackage{graphicx}%
\usepackage{hyperref}

 \newtheorem{Theorem}{Theorem}[section]
 \newtheorem{Corollary}[Theorem]{Corollary}
 \newtheorem{Lemma}[Theorem]{Lemma}
 \newtheorem{Proposition}[Theorem]{Proposition}

 \newtheorem{Remark}[Theorem]{Remark}

 \numberwithin{equation}{section}

\begin{document}

\title[strong openness conjecture]
 {Strong openness conjecture and related problems for plurisubharmonic functions}

\author{Qi'an Guan}
\address{Qi'an Guan: School of Mathematical Sciences, and Beijing International Center for Mathematical Research,
Peking University, Beijing, 100871, China.}
\email{guanqian@amss.ac.cn}
\author{Xiangyu Zhou}
\address{Xiangyu Zhou: Institute of Mathematics, AMSS, and Hua Loo-Keng Key Laboratory of Mathematics, Chinese Academy of Sciences, Beijing, China}
\email{xyzhou@math.ac.cn}

\thanks{The authors were partially supported by NSFC}

\subjclass{}

\keywords{$L^2$ extension theorem, strong openness conjecture,
plurisubharmonic function, multiplier ideal sheaf, complex
singularity exponent}

\date{}

\dedicatory{}

\commby{}


\begin{abstract}
In this article, we solve the strong
openness conjecture on the multiplier ideal sheaves for the
plurisubharmonic functions posed by Demailly.
We prove two conjectures about the growth of the volumes of the
sublevel sets of plurisubharmonic functions related to the complex singularity
exponents and quasi-plurisubharmonic functions related to
the jumping numbers,
which were posed by
Demailly-Koll\'{a}r and
Jonsson-Mustat\u{a} respectively.
We give a new proof
of a lower semicontinuity conjecture posed by
Demailly-Koll\'{a}r without using the ACC
conjecture.
Other applications by combining with well-known results
are also mentioned.
\end{abstract}

\maketitle

\section{Introduction}

Plurisubharmonic functions have been fundamental in several complex variables
and complex geometry since they were introduced by Oka and Lelong in 1940's.
For a nice survey on the theory of psh functions, the reader is referred to \cite{kisel}
by Kiselman. The philosophy behind the Levi problem, $L^2$ method for solving $\bar\partial$
equation and vanishing theorems on multiplier ideal sheaves is that construction
of a specific holomorphic function or section could be reduced to construction of a
speci¡¥c psh function. Singularities of the psh functions play an important role in
such a construction. In the present paper, we discuss the properties related to the
singularities of the psh functions.

\subsection{Outline of the main results and organizations}
$\\$

In this article, we establish a strong openness property for
plurisubharmonic functions.
We obtain estimates
about the growth of the volumes of the sublevel sets of plurisubharmonic
functions and quasi-plurisubharmonic
functions.
We also obtain a lower semicontinuity property of
plurisubharmonic functions.

We establish a strong openness property on multiplier ideal sheaves for plurisub-
harmonic functions. We obtain estimates about the growth of the volumes of the
sublevel sets of plurisubharmonic functions and quasi-plurisubharmonic functions.
We also obtain a lower semicontinuity property of plurisubharmonic functions.
The paper is organized as follows. In the rest of this section, we present our
main theorems and their corollaries, among others, solutions of the strong openness
conjecture posed by Demailly and two related conjectures posed by Demailly-Koll\'{a}r
and Jonsson-Mustat\u{a}. 
In Section 2, we recall or give some preliminary lemmas used
in the proofs of the main theorems. 
In Section 3, we introduce three propositions
used in the proofs of the main theorems. 
In Section 4, we give the proofs of the
propositions stated in the last Section. 
In Section 5, we give the detailed proofs of
the main theorems.

\subsection{Strong openness conjecture}\label{sub:strong}
$\\$

Let $X$ be complex manifold with dimension $n$ and $\varphi$ be a
plurisubharmonic function on $X$. Following Nadel \cite{Nadel90},
one can define the multiplier ideal sheaf $\mathcal{I}(\varphi)$ to
be the sheaf of germs of holomorphic functions $f$ such that
$|f|^{2}e^{-\varphi}$ is locally integrable (see also \cite{siu05},
\cite{siu09}, \cite{demailly2010}, etc.). Let
$$\mathcal{I}_{+}(\varphi):=\cup_{\varepsilon>0}\mathcal{I}((1+\varepsilon)\varphi).$$

In \cite{berndtsson13},
Berndtsson gave a proof of the openness conjecture of Demailly and Koll\'{a}r in \cite{D-K01}:

\emph{\textbf{Openness conjecture:} Let $\varphi$ be a plurisubharmonic function on $X$.
Assuming that $\mathcal{I}(\varphi)=\mathcal{O}_{X}$.
Then $$\mathcal{I}_{+}(\varphi)=\mathcal{I}(\varphi).$$}

The dimension two case of the Openness conjecture was proved by Favre and Jonsson in \cite{FM05j} (see also \cite{FM05v}).

In the present article, we discuss more general conjecture-the strong openness conjecture
about multiplier ideal sheaves for plurisubharmonic functions which was posed by
Demailly in
\cite{demailly-note2000} and \cite{demailly2010} (see also \cite{DEL00}, \cite{D-P03}, \cite{BFM08},
\cite{Leh11}, \cite{LiPhD}, \cite{JM12}, \cite{Gue12}, \cite{Cao12},
\cite{JM13}, \cite{Li2013}, \cite{Mat2013}, etc. ):

\emph{\textbf{Strong openness conjecture:}  Let $\varphi$ be a plurisubharmonic function on $X$.
Then $$\mathcal{I}_{+}(\varphi)=\mathcal{I}(\varphi).$$}

For $dim X \leq 2$, the strong openness conjecture was
proved in \cite{JM12} by studying the asymptotic jumping numbers for
graded sequences of ideals.

It is not hard to see that the truth of
the strong openness conjecture is equivalent to the following theorem:

\begin{Theorem}
\label{t:strong1015}\cite{GZopen-a}
Let $\varphi$ be a negative plurisubharmonic function on the unit polydisc $\Delta^{n}\subset\mathbb{C}^{n}$,
suppose $F$ is a holomorphic function on $\Delta^{n}$,
which satisfies
$$\int_{\Delta^{n}}|F|^{2}e^{-\varphi}d\lambda_{n}<+\infty,$$
where $d\lambda_{n}$
is the Lebesgue measure on $\mathbb{C}^{n}$.
Then there
exists a number $p>1$, such that
$$\int_{\Delta^{n}_{r}}|F|^{2}e^{-p\varphi}d\lambda_{n}<+\infty,$$
where $r\in(0,1)$.
\end{Theorem}

\subsection{A conjecture of Demailly and Koll\'{a}r}
$\\$

In \cite{D-K01},
Demailly and Koll\'{a}r posed a conjecture for the growth of the volumes of the
sublevel sets of plurisubharmonic functions related to the complex singularity
exponents
(see also \cite{FM05j}, \cite{FM05v}, \cite{JM12} and \cite{JM13}, etc.):

\emph{\textbf{Conjecture D-K:} Let $\varphi$ be a plurisubharmonic function on $\Delta^{n}\subset\mathbb{C}^{n}$,
and $K$ be compact subset of $\Delta^{n}$.
If $c_{K}(\varphi)<+\infty$,
then
$$\frac{1}{r^{2c_{K}(\varphi)}}\mu(\{\varphi<\log r\})$$
has
a uniform positive lower bound independent of $r\in(0,1)$,
where $c_{K}(\varphi)=sup\{c\geq0:\exp^{-2c\varphi}$ is $L^1$ on a neighborhood of $K\}$,
and $\mu$ is the Lebesgue volumes on $\mathbb{C}^{n}$.}

The above conjecture is a more precise form of
the openness conjecture.

By Theorem \ref{t:strong1015} and the fact that
$$c_{K}(\varphi)=\min_{z\in K}c_{\{z\}}(\varphi),$$
i.e. there exists $z\in K$, such that $c_{\{z\}}(\varphi)=c_{K}(\varphi)$ (see \cite{D-K01}),
then $e^{-2c_{K}(\varphi)\varphi}$ is not integrable on any neighborhood of $K$.

In the following theorem, we obtain an estimate for the sublevel sets
of plurisubharmonic functions:

\begin{Theorem}
\label{t:GZ_JM}
Let $\varphi$ be a plurisubharmonic function on $\Delta^{n}\subset\mathbb{C}^{n}$.
Let $F$ be a holomorphic function on $\Delta^{n}$.
Assume that $|F|^{2}e^{-\varphi}$ is not locally integrable near $o$,
Then
$$\int_{\Delta^{n}}\mathbb{I}_{\{-(R+1)<\varphi<-R\}}|F|^{2}e^{-\varphi}d\lambda_{n}$$
has a uniform positive lower bound independent of $R>>0$.
Especially, if $F=1$, then
$$e^{R}\mu(\{-(R+1)<\varphi<-R\})$$
has a uniform positive lower bound independent of $R>>0$.
\end{Theorem}

In particular,
and replacing $2c_{K}(\varphi)\varphi$ by $\varphi$, and $-2c_{K}\log r$ by $R$,
we solve the
Conjecture D-K:

\begin{Corollary}
\label{c:DK}
Conjecture D-K holds
\end{Corollary}

For $n\leq 2$, the above corollary
was proved by Favre and Jonsson in \cite{FM05j} (see also \cite{FM05v})

\subsection{A conjecture of Jonsson and Mustat\u{a}}
$\\$

Let $I$ be an ideal of $\mathcal{O}_{\Delta^{n},o}$, which is generated by $\{f_{j}\}_{j=1,\cdots,l}$.
In \cite{JM13},
Jonsson and Mustat\u{a} posed the following conjecture about the volumes growth of the
sublevel sets of quasi-plurisubharmonic functions
(see also \cite{JM12}):

\emph{\textbf{Conjecture J-M:} Let $\psi$ be a plurisubharmonic function on $\Delta^{n}\subset\mathbb{C}^{n}$.
If $c_{o}^{I}(\psi)<+\infty$,
then
$$\frac{1}{r^2}\mu(\{c^{I}_{o}(\psi)\psi-\log|I|<\log r\})$$
has
a uniform positive lower bound independent of $r\in(0,1)$,
where
$$\log|I|:=\log\max_{1\leq j\leq l}|f_{j}|,$$
$c_{o}^{I}(\psi)=sup\{c\geq0:|I|^{2}e^{-2c\psi}$ is $L^1$ on a neighborhood of $o\}$ is the jumping number in \cite{JM13},
and $\mu$ is the Lebesgue measure on $\mathbb{C}^{n}$.}

For $n\leq 2$, the above conjecture
was proved by Jonsson and Mustat\u{a} in \cite{JM12}.

In the following theorem, we give an estimate for the sublevel sets
of quasiplurisubharmonic functions:

\begin{Theorem}
\label{t:GZ_JM201312}
Let $\psi$ be a plurisubharmonic function on $\Delta^{n}$,
and $F$ be a holomorphic function on $\Delta^{n}$.
Assume that $|F|^{2}e^{-\psi}$ is not locally integrable near $o$.
Then
$$e^{R}\frac{1}{B_{0}}\mu(\{-R-B_{0}<\psi-\log|F|^{2}<-R\})$$
has
a uniformly positive lower bound independent of $R>>0$ and $B_{0}\in(0,1]$.
\end{Theorem}

In particular, we obtain:
\begin{equation}
\label{equ:2014113b}
e^{R}\mu(\{\psi-\log|F|^{2}<-R\})
\end{equation}
has
a uniform positive lower bound independent of $R\in(0,+\infty)$.

By Theorem \ref{t:strong1015},
it follows that $|I|^{2}e^{-2c_{o}^{I}(\psi)\psi}$ is not integrable on any neighborhood of $o$.
Replacing $\psi$ by  $2c_{o}^{I}(\psi)\psi$, and  $R$ by $-2\log r$ in equality \ref{equ:2014113b},
we solve conjecture J-M:

\begin{Corollary}
Conjecture J-M holds.
\end{Corollary}

\subsection{A lower semicontinuity conjecture}
$\\$

In \cite{D-K01},
Demailly and Koll\'{a}r conjectured that:

For every nonzero holomorphic function $f$ on $X$,
there is a number $\delta=\delta(f,K,L)>0$, such that for any holomorphic function
$g$ on $X$ with
$$\sup_{L}|g-f|<\delta\Rightarrow c_{K}(\log|g|)\geq c_{K}(\log|f|),$$
where the compact set $K$ contained in an open subset $L$ of complex manifold $X$.

In \cite{D-K01},
the authors proved that
the above conjecture
is implied by the ACC conjecture (see \cite{Sho92} or \cite{Ko92}).
The ACC conjecture was proved by Hacon, McKernan and Xu in \cite{HKX2012}.

It is not hard to see that the just mentioned conjecture posed by Demailly and
Koll\'{a}r is equivalent to the following conjecture:

Let $\{g_{m}\}_{m=1,2,\cdots}$ be a sequence of holomorphic functions on $\Delta^{n}$,
which are uniformly convergent to holomorphic function $f$ on $\Delta^{n}$.
Assume that $|g_{m}|^{-2c}$ is not integrable near $o\in\Delta^{n}$ for any $m=1,2,\cdots$,
where $c$ is a positive constant.
Then $|f|^{-2c}$ is not integrable near $o\in\Delta^{n}$.

Note that $c\log|f|$ is a plurisubharmonic function,
then
we replace $c\log|f|$ by general plurisubharmonic functions,
and
obtain the following lower semicontinuous property of plurisubharmonic functions:

\begin{Proposition}
\label{t:GZ_open1026}
Let $\{\phi_{m}\}_{m=1,2,\cdots}$ be a sequence of negative plurisubharmonic functions on $\Delta^{n}$,
which is convergent to a negative Lebesgue measurable function $\phi$ on $\Delta^{n}$ in Lebesgue measure.
Assume that $e^{-\phi_{m}}$ are all not integrable near $o$.
Then $e^{-\phi}$ is not integrable near $o$.
\end{Proposition}

Replacing
$c\log|f|$ by any plurisubharmonic functions in the above conjecture,
using Proposition \ref{t:GZ_open1026},
we obtain a generalization of the conjecture as follows,
which is a new proof of the conjecture without using the ACC conjecture:

Let $\{\phi_{m}\}_{m=1,2,\cdots}$ be a sequence of plurisubharmonic functions on $\Delta^{n}$,
such that $e^{\phi_{m}}$ are uniformly convergent to $e^{\phi}$ on $\Delta^{n}$,
where $\phi$ is a plurisubharmonic function on on $\Delta^{n}$.
Assume that $e^{-\phi_{m}}$ is not integrable near $o\in\Delta^{n}$ for any $m=1,2,\cdots$.
Then $e^{-\phi}$ is not integrable near $o\in\Delta^{n}$.

\subsection{Some applications of the strong openness conjecture}
$\\$

In the present subsection, combining our Theorem \ref{t:strong1015} (the truth of the strong openness conjecture) with some known results,
we obtain some direct conclusions.

\subsubsection{Singular metric with minimal singularities}
$\\$

Let $L$ be a line bundle on a smooth projective complex variety $X$,
whose Kodaira-Iitaka dimension
$\kappa(X, L)\geq0$ (see \cite{Lazar04I,Lazar04II,DEL00,Leh11}).
Then the asymptotic multiplier ideal $\mathcal{J}(||L||)$ can be
defined as the maximal member of the family of ideals
$\{\mathcal{J}(\frac{1}{k}\cdot |kL|)\}$ ($k$ large) (see Definition 1.7 in \cite{DEL00}, see also \cite{Lazar04II}).

Demailly has shown that if $L$ is any pseudo-effective divisor,
then up to equivalence of singularities, $\mathcal{O}_{X}(L)$ has a unique singular metric $h_{\min}$
with minimal singularities having non-negative curvature current (see \cite{demailly2010}, see also \cite{Lazar04II}).

Let $\mathcal{J}(h_{\min})$ be the associated multiplier
ideal sheaf of $h_{\min}$ (see \cite{demailly2010}).

In \cite{DEL00,Lazar04II},
the authors conjectured the following:

For big line bundle $L$,
the
equality
\begin{equation}
\label{equ:2014113c}
\mathcal{J}(||mL||)=\mathcal{J}(h^{m}_{\min})
\end{equation}
holds
for every $m>0$.

Note that $mL$ is big and $\kappa(X, mL)\geq0$ for any $m$.
Let $\tilde{h}^{m}_{\min}$ be the singular metric with minimal singularities on $L^{m}$.
By the uniqueness of the singular metric with minimal singularities,
it follows that $\frac{\tilde{h}^{m}_{\min}}{h^{m}_{\min}}$ is a function with uniformly positive upper and lower bound on $X$.
Therefore $\mathcal{J}(\tilde{h}^{m}_{\min})=\mathcal{J}(h^{m}_{\min})$.
Then it suffices to consider the case $m=1$.

In \cite{Leh11}, the author conjectured the following analogue of the above conjecture:

Let $X$ be a smooth projective complex variety and
$L$ be a pseudo-effective $\mathbb{R}-$divisor on $X$. Then
\begin{equation}
\label{equ:2014113d}
\mathcal{J}(T_{\min})\subseteq\mathcal{J}_{\sigma}(L),
\end{equation}
where $T_{\min}$ is a current of minimal singularities in the numerical class of $L$,
and $\mathcal{J}_{\sigma}(L)$ is the diminished ideal in \cite{Leh11}.

By the arguments after Theorem 1.2 in \cite{Leh11}, the above conjecture can be proved by the strong openness conjecture.

In \cite{Leh11},
it was shown that
$$\mathcal{J}_{\sigma}(L)=\mathcal{J}(||L||)$$
when $L$ is a big line bundle (see Corollary 6.12 in \cite{Leh11}).
Note that
$\mathcal{J}(h_{\min})$ in \cite{DEL00} is just $\mathcal{J}(T_{\min})$ in \cite{Leh11}.
Using inequality \ref{equ:2014113d},
one has
$$\mathcal{J}(h_{\min})\subseteq\mathcal{J}(||L||).$$

In \cite{DEL00},
it was shown that
$$\mathcal{J}(||L||)\subseteq\mathcal{J}(h_{\min}).$$
Then one has equality \ref{equ:2014113c} holds.

\subsubsection{Kawamata-Viehweg-Nadel type
vanishing theorem}
$\\$

Let $(L,\varphi)$ be a pseudo-effective line bundle on a compact
K\"{a}hler manifold $X$ of dimension $n$, and $nd(L,\varphi)$ be the
numerical dimension of $(L,\varphi)$ defined in \cite{Cao12}.

In
\cite{Cao12}, Cao obtained a Kawamata-Viehweg-Nadel type
vanishing theorem for $\mathcal{I}_{+}(\varphi)$ on any
compact K\"{a}hler manifold:
$$H^{p}(X,K_{X}\otimes L\otimes\mathcal{I}_{+}(\varphi))=0$$
holds for any $p\geq n-nd(L,\varphi)+1$.

In \cite{Cao12}, Cao asked whether the
Kawamata-Viehweg-Nadel type vanishing theorem holds for $\mathcal{I}(\varphi)$ on any compact
K\"{a}hler manifold, i.e.
does
$$H^{p}(X,K_{X}\otimes L\otimes\mathcal{I}(\varphi))=0$$
hold
for any $p\geq n-nd(L,\varphi)+1$?

Using Theorem \ref{t:strong1015} and the above Kawamata-Viehweg-Nadel type
vanishing theorem for $\mathcal{I}_{+}(\varphi)$,
one can answer the just mentioned question of Cao as follows:
\begin{Corollary}
Let $(L,\varphi)$ be a pseudo-effective line bundle on a compact K\"{a}hler manifold $X$ of dimension $n$,
Then
$$H^{p}(X,K_{X}\otimes L\otimes\mathcal{I}(\varphi))=0,$$
for any $p\geq n-nd(L,\varphi)+1$.
\end{Corollary}
When $X$ is a projective manifold,
some similar result can be referred to \cite{Mat2013b}.

\subsubsection{Multiplier ideal sheaves with analytic singularities}
$\\$

It is known that $\mathcal{I}_{+}(\varphi)$
is essentially with analytic singularities (see \cite{Cao12})
using Demailly's approximation of plurisubharmonic functions (see \cite{demailly2010}).
Then it follows from
Theorem \ref{t:strong1015} that $\mathcal{I}(\varphi)$ is
essentially with analytic singularities, that is to say,

\begin{Corollary}
There is a
plurisubharmonic function $\varphi_A$ with analytic singularities,
such that $\mathcal{I}(\varphi)=\mathcal{I}(\varphi_A)$ .
\end{Corollary}

\section{Some lemmas used in the proof of main theorems}

In this section,
we will show some results used in the proof of main theorem.

\subsection{$L^{1}$ integrable function}
$\\$

Let $G$ be a positive Lebesgue measurable and integrable function on a domain $\Omega\subset\subset\mathbb{C}^{n}$,
which is
$$\int_{\Omega}Gd\lambda_{n}<+\infty.$$

Consider the function
$$F_{G}(t):=\sup\{a|\lambda_{n}(\{G\geq a\})\geq t\},$$
$t\in(0,\lambda_{n}(\Omega)].$

We first consider the finiteness of $F_{G}(t)$:

If for some $t_{0}$, $F_{G}(t_{0})=+\infty$,
then $\lambda_{n}(G\geq A_{j})\geq t_{0}$,
where $A_{j}$ is a number sequence tending to $+\infty$,
when $j\to+\infty$.
Since $G$ is $L^1$ integrable, we have
$$t_{0}A_{j}\leq A_{j}\lambda_{n}(\{G\geq A_{j})\leq\int_{\{G\geq A_{j}\}}Gd\lambda_{n}\leq\int_{\Omega}Gd\lambda_{n}<+\infty,$$
Letting $A_{j}\to+\infty$,
we thus obtain a contradiction.
Therefore $F_{G}(t)<+\infty$ for any $t$.

Secondly,
we consider the decreasing property of $F_{G}(t)$:

Note that $\{A|\lambda_{n}(G\geq A)\geq t_{1}\}\supset\{A|\lambda_{n}(G\geq A)\geq t_{2}\}$,
when $t_{1}\leq t_{2}$.
Then we have $F_{G}(t_{1})\geq F_{G}(t_{2})$, when $t_{1}\leq t_{2}$.

The first lemma is about the sublevel sets of $F_{G}$:

\begin{Lemma}
\label{l:level}
We have
\begin{equation}
\label{equ:level}
\mu_{\mathbb{R}}(\{t|F_{G}(t)\geq a\})=\lambda_{n}(\{G\geq a\}),
\end{equation}
for any $a>0$,
where $\mu_{\mathbb{R}}$ is the Lebesgue measure on $\mathbb{R}$.
Moreover, we have
$$\mu_{\mathbb{R}}(\{t|F_{G}(t)> a\})=\lambda_{n}(\{G> a\}).$$
\end{Lemma}

\begin{proof}
Since
$$\mu_{\mathbb{R}}(\{t|F_{G}(t)> a\})=\lim_{k\to+\infty}\mu_{\mathbb{R}}(\{t|F_{G}(t)\geq a+\frac{1}{k}\}),$$
and
$$\lambda_{n}(\{G> a\})=\lim_{k\to+\infty}\lambda_{n}(\{G\geq a+\frac{1}{k}\}),$$
we only need to prove
$$\mu_{\mathbb{R}}(\{t|F_{G}(t)\geq a\})=\lambda_{n}(\{G\geq a\}),$$
for any $a>0$,

Note that $$\sup\{a_{1}|\lambda_{n}(\{G\geq a_{1}\})\geq\lambda_{n}(\{G\geq a\})\}\geq a.$$
Then we have
$$\lambda_{n}(\{G\geq a\})\in\{t|\sup\{a_{1}|\lambda_{n}(\{G\geq a_{1}\})\geq t\}\geq a\},$$
therefore
$$\{t|\sup\{a_{1}|\lambda_{n}(\{G\geq a_{1}\})\geq t\}\geq a\}\supseteq\{t|\lambda_{n}(\{G\geq a\})\geq t\},$$
where $t>0$.

If $"\supseteq"$ in the above relation is strictly $"\supset"$,
then there exists $t_{0}$, such that

1). $\sup\{a_{1}|\lambda_{n}(\{G\geq a_{1}\})\geq t_{0}\}\geq a$;

2). $t_{0}>\lambda_{n}(\{G\geq a\})$.

Let
$$a_{0}:=\sup\{a_{1}|\lambda_{n}(\{G\geq a_{1}\})\geq t_{0}\}\geq a.$$
Note that
$$\lambda_{n}(\cap_{a_{1}<a_{0}}\{G\geq a_{1}\})=\inf_{a_{1}<a_{0}}\lambda_{n}(\{G\geq a_{1}\}).$$
Then we have $\lambda_{n}(\{G\geq a_{0}\})\geq t_{0}$.

As $t_{0}>\lambda_{n}(\{G\geq a\})$,
we have
$$\lambda_{n}(\{G\geq a_{0}\})>\lambda_{n}(\{G\geq a\}),$$
which is a contradiction to
$$a_{0}\geq a.$$

Then the following holds:
\begin{equation}
\label{equ:0807a}
\{t|\sup\{a_{1}|\lambda_{n}(\{G\geq a_{1}\})\geq t\}\geq a\}=\{t|\lambda_{n}(\{G\geq a\})\geq t\},
\end{equation}
where $t>0$.

According to the definition of $F_{G}$ and equality \ref{equ:0807a},
it follows that
$$\{t|F_{G}(t)\geq a\}
=\{t|\sup\{a_{1}|\mu(\{G\geq a_{1}\})\geq t\}\geq a\}=\{t|\lambda_{n}(\{G\geq a\})\geq t\},$$
where $t>0$.

Note that
$$\mu_{\mathbb{R}}(\{t|F_{G}(t)\geq a\})=\mu_{\mathbb{R}}\{t|\lambda_{n}(\{G\geq a\})\geq t\}=\lambda_{n}(\{G\geq a\}),$$
for $t>0$.
We have thus proved the present lemma.
\end{proof}

Denote by
$$s(y):=y^{-1}(-\log y)^{-1},$$
where $y\in (0,e^{-1})$.
It is clear that $s$ is strictly decreasing on $(0,e^{-1})$.

We define a function $u$ by
$$u(s(y))=y^{-1},$$
where
$u\in C^{\infty}((e,+\infty))$.
It is clear that $u$ is strictly increasing on $(e,+\infty)$,

The second lemma is about the measure of the level set of $G$:

\begin{Lemma}
\label{l:open_a}
For $A>e$, we have
$$\liminf_{A\to+\infty}\lambda_{n}(\{G> A\})u(A)=0,$$
where
$\lambda_{n}$ is the Lebesgue measure of $\mathbb{C}^{n}$.
Especially,
$\lim_{A\to+\infty}\frac{A}{u(A)}=0$.
\end{Lemma}

\begin{proof}
According to the definition of Lebesgue integration and Lemma \ref{l:level},
it follows that
$$\int_{0}^{\mu(\Omega)}F_{G}(t)dt=\int_{\Omega}Gd\lambda_{n}<+\infty.$$
Then we have
$$\liminf_{t\to 0}\frac{F_{G}(t)}{t^{-1}(-\log t)^{-1}}=0,$$
which implies that there exists $t_{j}\to 0$, when $j\to+\infty$,
such that
\begin{equation}
\label{equ:lim_F_G}
\lim_{j\to+\infty}\frac{F_{G}(t_{j})}{t_{j}^{-1}(-\log t_{j})^{-1}}=0.
\end{equation}

Using Lemma \ref{l:level},
we have
$$\lambda_{n}(\{G> F_{G}(t_{j})\})=\mu_{\mathbb{R}}(\{t|F_{G}(t)> F_{G}(t_{j})\})\leq\mu_{\mathbb{R}}((0,t_{j}))=t_{j}.$$

We now want to prove that $u(F_{G}(t_{j}))=o(t_{j}^{-1})$ by contradiction:
if not, there exists $\varepsilon_{0}>0$,
such that $u(F_{G}(t_{j}))\leq \varepsilon_{0}t_{j}^{-1}$.
However,
$$u(F_{G}(t_{j}))\leq \varepsilon_{0}t_{j}^{-1}=u(\frac{1}{\frac{t_{j}}{\varepsilon_{0}}(-\log\frac{t_{j}}{\varepsilon_{0}})}).$$

According to the strictly increasing property of $u$,
it follows that $F_{G}(t_{j})\leq \frac{1}{\frac{t_{j}}{\varepsilon_{0}}(-\log\frac{t_{j}}{\varepsilon_{0}})}$,
which is contradict to equality \ref{equ:lim_F_G}
because
$\lim_{t_{j}\to0}\frac{t_{j}(-\log t_{j})}{\frac{t_{j}}{\varepsilon_{0}}(-\log\frac{t_{j}}{\varepsilon_{0}})}=\varepsilon_{0}$.
Now we obtain $u(F_{G}(t_{j}))=o(t_{j}^{-1})$.

Then we have
$$\lim_{j\to+\infty}\mu(\{G> F_{G}(t_{j})\})u(F_{G}(t_{j}))\leq\lim_{j\to+\infty}t_{j}o(t_{j}^{-1})=0.$$

Note that if $F_{G}(t_{j})$ is bounded above,
when $t_{j}$ to $0$,
then $G$ has positive upper bound.
Therefore $\mu(\{G> A\})=0$, for $A$ large enough.

Then we have proved
$$\liminf_{A\to+\infty}\lambda_{n}(\{G> A\})u(A)=0.$$

As $\frac{1}{t(-\log t)}$ is strictly decreasing on $(0,e^{-1})$,
then for any $A>e$,
there exists $t_{A}$,
such that

1). $\frac{1}{t_{A}(-\log t_{A})}=A$;

2). $t_{A}$ goes to zero, when $A$ goes to $+\infty$.

As
$$\frac{A}{u(A)}=\frac{\frac{1}{t_{A}(-\log t_{A})}}{u(\frac{1}{t_{A}(-\log t_{A})})}=\frac{\frac{1}{t_{A}(-\log t_{A})}}{\frac{1}{t_{A}}}=\frac{1}{-\log t_{A}},$$
then we obtain
$$\lim_{A\to+\infty}\frac{A}{u(A)}=0,$$
by the above property $2)$ of $t_{A}$.

The present lemma is thus proved.

\end{proof}

\subsection{Estimation of integration of holomorphic functions on singular Riemann
surfaces}

\begin{Lemma}
\label{l:open_b}
Let $h\not\equiv0$ be a holomorphic function on the disc $\Delta_{r}$ in $\mathbb{C}$.
Let $f_{a}$ be a holomorphic function on $\Delta_{r}$,
which satisfies
$f|_{o}=0$ and $f_{a}(b)=1$ for any $b^{k}=a$ ($k$ is a positive integer),
then we have
$$\int_{\Delta_{r}}|f_{a}|^{2}|h|^{2}d\lambda_{1}>C_{1}|a|^{-2},$$
where $a\in\Delta_{r}$ whose norm is small enough,
$k$ is a positive integer,
$C_{1}$ is a positive constant independent of $a$ and $f_{a}$.
\end{Lemma}

\begin{proof}
As $h\not\equiv0$,
we may write $h=z^{i}h_{1}$ near $o$,
where $h_{1}|_{o}\neq 0$.
Then there exists $r'<r$,
such that $h_{1}|_{\Delta_{r'}}\geq C_{0}>0$.
Therefore it suffices to consider the case that $h=z^{i}$ on $\Delta_{r'}$.

By Taylor expansion at $o$,
we have
$$f(z)=\sum_{j=1}^{\infty}c_{j}z^{j}.$$

As $f(b)=1$,
then
$$\sum_{j=1}^{\infty}c_{kj}a^{j}=\frac{1}{k}\sum_{1\leq l\leq k}\sum_{j=1}^{\infty}c_{j}b_{l}^{j}=1$$
where $b_{l}^{k}=a$,
and $\sum_{1\leq l\leq k} b_{l}^{j}=0$ when $0<j<k$.

It is clear that
\begin{equation}
\label{equ:0808a}
\int_{\Delta_{r'}}|f_{a}|^{2}|h|^{2}d\lambda_{1}=\int_{\Delta_{r'}}|f_{a}|^{2}|z^{i}|^{2}d\lambda_{1}=
2\pi\sum_{j=1}^{\infty}|c_{j}|^{2}\frac{{r'}^{2j+2i+2}}{2j+2i+2}
\end{equation}

By Schwartz Lemma,
we have
\begin{equation}
\label{}
\begin{split}
&(\sum_{j=1}^{\infty}|c_{j}|^{2}\frac{{r'}^{2j+2i+2}}{2j+2i+2})(\sum_{j=1}^{\infty}\frac{2kj+2i+2}{{r'}^{2kj+2i+2}}|a|^{2j})
\\&\geq(\sum_{j=1}^{\infty}|c_{kj}|^{2}\frac{{r'}^{2kj+2i+2}}{2kj+2i+2})(\sum_{j=1}^{\infty}\frac{2kj+2i+2}{{r'}^{2kj+2i+2}}|a|^{2j}) \geq|\sum_{j=1}^{\infty}c_{kj}a^{j}|^{2}=1.
\end{split}
\end{equation}

Note that
$$\sum_{j=1}^{\infty}\frac{2kj+2i+2}{{r'}^{2kj+2i+2}}|a|^{2j}
=|\frac{a}{{r'}^{k}}|^{2}((2i+2)\frac{{r'}^{-2i-2}}{1-|\frac{a}{{r'}^{k}}|^{2}}+2k\frac{{r'}^{-2i-2}}{(1-|\frac{a}{{r'}^{k}}|^{2})^{2}}),$$
and $((2i+2)\frac{{r'}^{-2i-2}}{1-|\frac{a}{{r'}^{k}}|^{2}}+2k\frac{{r'}^{-2i-2}}{(1-|\frac{a}{{r'}^{k}}|^{2})^{2}})$ has uniform upper bound
independent of $a$,
when $|a|<\frac{{r'}^{k}}{2}$.
The Lemma thus follows.

\end{proof}

Let¡¯s recall the local parametrization theorem:

\begin{Theorem}
\label{t:para}(see \cite{demailly-book})
Let $\mathscr{J}$ be a prime ideal of $\mathcal{O}_{n}$ and
let $\mathcal{C}'=V(\mathscr{J})$ be an analytic curve at $o$.
Then the ring $\mathcal{O}_{n}/\mathscr{J}$ is a finite integral extension of $\mathcal{O}_{d}$;
let $q$ be the degree of the extension.
There exists a local coordinates
$$(z';z'')=(z_{1};z_{2},\cdots,z_{n}),$$
such that
if $\Delta'_{r'}$ and $\Delta''_{r''}$ are polydisks of sufficient small
radii $r'$ and $r''$ and if $r'\leq \frac{r''}{C}$ with $C$ large,
the projective map $\pi':\mathcal{C}'\cap (\Delta'_{r'}\times\Delta''_{r''})\to\Delta'_{r'}$
is a ramified covering with $q$ sheets, whose ramification locus
is contained in $S=\{o'\}\subset\Delta'_{r'}$.
This means that

a), the open set $\mathcal{C}'_{S}:=\mathcal{C}'\cap ((\Delta'_{r'}\setminus S)\times\Delta''_{r''})$
is a smooth $1-$dimensional manifold, dense in $\mathcal{C}'\cap(\Delta'_{r'}\times\Delta''_{r''})$;

b), $\pi':\mathcal{C}'_{S}\to\Delta'_{r'}\setminus S$ is an unramified covering;

c), the fibre $\pi^{'-1}(z')$ have exactly $q$ elements if $z'\in \Delta'\setminus S$ and at most $q$ if $z'\in S$.

Moreover, $\mathcal{C}'_{S}$ is a connected covering of $\Delta'_{r'}\setminus S$,
and $\mathcal{C}'\cap(\Delta'_{r'}\times\Delta''_{r''})$ is contained in a cone $|z''|\leq\frac{C}{6}|z'|$.
\end{Theorem}

Let $\Delta'$ and $\Delta''$ be unit disc with coordinates
$(z_{1})$ and unit polydisc with coordinates $(z_{2},\cdots,z_{n})$ respectively.

Let
$$\pi:\Delta'\times\Delta''\to \Delta'$$
be the projective map which is given by
$$\pi(z_{1};z_{2},\cdots,z_{n})=z_{1}.$$

We have the following Remark of Theorem \ref{t:para}.

\begin{Remark}
\label{r:para}
Let $\mathscr{J}$ be a prime ideal of $\mathcal{O}_{n}$ and
let $\mathcal{C}'=V(\mathscr{J})$ be an analytic curve at $o$.
Then the ring $\mathcal{O}_{n}/\mathscr{J}$ is a finite integral extension of $\mathcal{O}_{d}$;
let $q$ be the degree of the extension,
there exists a biholomorphic map $j$ from
a neighborhood of $\overline{\Delta'\times\Delta''}$
to a neighborhood  $U_{o}$ of $o$,
such that the projective map $\pi|_{\mathcal{C}\cap (\Delta'\times\Delta'')}\to\Delta'$
is a ramified covering with $q$ sheets, whose ramification locus
is contained in $S=\{o'\}\subset\Delta'$
where
$$\mathcal{C}:=j^{-1}(\mathcal{C}').$$
This means that

a), the open set $\mathcal{C}_{S}:=\mathcal{C}\cap ((\Delta'\setminus S)\times\Delta'')$
is a smooth $1-$dimensional manifold, dense in $\mathcal{C}\cap(\Delta'\times\Delta'')$;

b), $\pi|_{\mathcal{C}_{S}}:\mathcal{C}_{S}\to\Delta'\setminus S$ is an unramified covering;

c), the fibre $\pi^{-1}(z')$ have exactly $q$ elements if $z'\in\Delta'\setminus S$ and at most $q$ if $z'\in S$.

Moreover, $\mathcal{C}_{S}$ is a connected covering of $\Delta'\setminus S$,
and $\mathcal{C}\cap(\Delta'\times\Delta'')$ is contained in a cone $|z''|\leq\frac{1}{6}|z'|$.

\end{Remark}

Using Lemma \ref{l:open_b} and Remark \ref{r:para},
we obtain the following singular version of
Lemma \ref{l:open_b}:

\begin{Lemma}
\label{l:open_sing}
Let $h$ be a holomorphic function on an analytic curve $\mathcal{C}$ as in Remark \ref{r:para}.
Let $f_{a}$ be a holomorphic function on $\mathcal{C}$,
which satisfies
$f(o)=0$ and $f_{a}(\pi^{-1}(a)\cap\mathcal{C})=1$,
then we have
$$\int_{\mathcal{C}_{S}}|f_{a}|^{2}|h|^{2}(\pi|_{\mathcal{C}_{S}})^{*}d\lambda_{\Delta'}>C_{2}|a|^{-2},$$
when $a\in\Delta'$ and $|a|$ is small enough,
where
$C_{2}$ is a positive constant independent of $a$ and $f_{a}$.
\end{Lemma}

\begin{proof}
As $(\mathcal{C},o)$ is irreducible and locally irreducible,
then there is a normalization $j_{nor}:(\Delta,0)\to(\mathcal{C},o)$,
denoted by
$$j_{nor}(t)=(g_{1}(t),\cdots,g_{n}(t)),$$
where $t$ is the coordinate of $\Delta$.
As $\pi_{\mathcal{C}}$ is a covering,
then $g_{1}\not\equiv0$.

Without loss of generality,
we assume $g_{1}(t)=t^{i_{1}}$ on $\Delta_{r_0}$,
for small enough $r_{0}\in(0,1)$.

There is a given $r>0$, which is small enough,
such that
$$(\mathcal{C}\cap\Delta^{n}_{r_{0}})\supset\{(t^{i_1},g_{2}(t),\cdots,g_{n}(t))|t\in\Delta_{r}\},$$
where $i_{1}\geq 1$, and $g_{i}$ ($i\geq 2$) are holomorphic functions on $\Delta_{r}$,
satisfying $|g_{i}|\leq \frac{1}{6} |t^{i_1}|$.

For given $r'<r$ small enough,
we have
\begin{equation}
\label{equ:0808b}
\begin{split}
\int_{\mathcal{C}_{S}}|f_{a}|^{2}|h|^{2}(\pi|_{\mathcal{C}_{S}})^{*}d\lambda_{\Delta'}
&\geq i_{1}^{2}\int_{\Delta_{r'}}|j_{nor}^{*}f_{a}(t)|^{2}|j_{nor}^{*}h(t)|^{2}|t|^{2(i_{1}-1)}d\lambda_{\Delta}
\\&=i_{1}^{2}\int_{\Delta_{r'}}|j_{nor}^{*}f_{a}(t)(t^{i_{1}-1}j_{nor}^{*}h(t))|^{2}d\lambda_{\Delta},
\end{split}
\end{equation}
for any $a$ satisfying $|a|^{\frac{1}{i_1}}\in\Delta_{r'}$.

As
$$j_{nor}^{*}f_{a}(b)=f_{a}(b^{i_1},g_{2}(b),\cdots,g_{n}(b))$$
and
$$(b^{i_1},g_{2}(b),\cdots,g_{n}(b))\subset(\pi^{-1}(b^{i_1})\cap\mathcal{C}),$$
then we have
$$j_{nor}^{*}f_{a}(b)=1,$$
for any $b^{i_{1}}=a$.

Using Lemma \ref{l:open_b},
we have
$$\int_{\Delta_{r'}}|j_{nor}^{*}f_{a}(t)(t^{i_{1}-1}j_{nor}^{*}h(t))|^{2}d\lambda_{\Delta}\geq C_{1}|a|^{-2},$$
where $C_{1}$ is independent of $a$ and $f_{a}$.

Combining with inequality \ref{equ:0808b}, we thus obtain the present lemma.
\end{proof}

As $\mathcal{C}\cap(\Delta'\times\Delta'')$ is contained in a cone $|z''|\leq\frac{1}{6}|z'|$,
using the submean value
property of plurisubharmonic function,
we obtain the following lemma:

\begin{Lemma}
\label{l:approx.L2}
For any holomorphic function $F$ on $\Delta'\times\Delta''$,
we obtain an approximation of the $L^{2}$ norm of $F$:

$$\int_{\Delta'\times\Delta''}|F|^{2}d\lambda_{n}\geq
C_{3}\int_{\mathcal{C}_{S}}|F|_{\mathcal{C}_{S}}|^{2}(\pi|_{\mathcal{C}_{S}})^{*}d\lambda_{\Delta'},$$
where $C_{3}$ is a positive constant independent of $F$.
Here all symbols $\mathcal{C}_{S}$, $\Delta'$ and $\pi$ are the same as in Remark \ref{r:para}.
\end{Lemma}

\begin{proof}
Using the Fubini Theorem,
$$\int_{\Delta'\times\Delta''}|F|^{2}d\lambda_{n}=
\int_{\Delta'}(\int_{\{z'\}\times\Delta''}|F|^{2}d\lambda_{n-1})d\lambda_{\Delta'},$$
and the submean value inequality of plurisubharmonic function,
we have
$$\int_{\{z'\}\times\Delta''}|F|^{2}d\lambda_{n-1}\geq(\frac{\pi}{3})^{n-1}|F(z',z'')|^{2},$$
for $|z''|\leq\frac{1}{6}$.

If $(z',z'')\in(\pi^{-1}(z')\cap\mathcal{C}_{S})$,
then $|z''|\leq \frac{1}{6}$.

As
$$\int_{\Delta'\setminus \{0\}}(\sum_{w\in(\pi^{-1}(z')\cap\mathcal{C}_{S})}|F(w)|^{2})(z')d\lambda_{\Delta'}
=\int_{\mathcal{C}_{S}}|F|_{\mathcal{C}_{S}}|^{2}(\pi|_{\mathcal{C}_{S}})^{*}d\lambda_{\Delta'},$$
it follows that
\begin{equation}
\label{equ:0922.1}
\begin{split}
q\int_{\Delta'\times\Delta''}|F|^{2}d\lambda_{n}=
&q\int_{\Delta'\setminus \{0\}}(\int_{\{z'\}\times\Delta''}|F|^{2}d\lambda_{n-1})d\lambda_{\Delta'}
\\\geq& (\frac{\pi}{3})^{n-1}\int_{\Delta'\setminus \{0\}}(\sum_{w\in(\pi^{-1}(z')\cap\mathcal{C}_{S})}|F(w)|^{2})(z')d\lambda_{\Delta'}
\\=&(\frac{\pi}{3})^{n-1}\int_{\mathcal{C}_{S}}|F|_{\mathcal{C}_{S}}|^{2}(\pi|_{\mathcal{C}_{S}})^{*}d\lambda_{\Delta'},
\end{split}
\end{equation}
where $q$ is the degree of the covering map $\pi|_{\mathcal{C}_{S}}$.
\end{proof}

\subsection{$L^{2}$ extension theorem with negligible weight.}
$\\$

We state the optimal constant version of the Ohsawa's $L^2$ extension theorem with negligible weight (\cite{ohsawa3}) as follows:

\begin{Theorem}\label{t:guan-zhou12}
\cite{guan-zhou12}
Let $X$ be a Stein manifold of dimension n. Let $\varphi+\psi$ and
$\psi$ be plurisubharmonic functions on $X$. Assume that $w$ is a
holomorphic function on $X$ such that
$\sup\limits_X(\psi+2\log|w|)\leq0$ and $dw$ does not vanish
identically on any branch of $w^{-1}(0)$. Put $H=w^{-1}(0)$ and
$H_0=\{x\in H:dw(x)\neq 0\}$. Then  there exists a uniform constant
$\mathbf{C}=1$ independent of $X$, $\varphi$, $\psi$ and $w$ such
that, for any holomorphic $(n-1)$-form $f$ on $H_0$ satisfying
$$c_{n-1}\int_{H_0}e^{-\varphi-\psi}f\wedge\bar f<\infty,$$
where
$c_{k}=(-1)^{\frac{k(k-1)}{2}}(\sqrt{-1})^k$ for $k\in\mathbb{Z}$,
there exists a holomorphic $n$-form F on $X$ satisfying $F=dw\wedge
\tilde{f}$ on $H_0$ with $\imath^*\tilde{f}=f$ and
$$c_n\int_{X}e^{-\varphi}F\wedge\bar F\leq 2\mathbf{C}\pi
c_{n-1}\int_{H_0}e^{-\varphi-\psi}f\wedge\bar f,$$
where
$\imath:H_0\longrightarrow X$ is the inclusion map.
\end{Theorem}

\subsection{Curve selection lemma and Noetherian property of coherent sheaves}
$\\$

We give the existence of some kind of germs of analytic curves,
which will be
used.

\begin{Lemma}
\label{l:curve_exist}
Let $(Y,o)$ be a germ of irreducible analytic subvariety in $\mathbb{C}^{n}$,
and $(A,o)$ be a germ of analytic subvariety of $(Y,o)$,
such that $dim A<dim Y$.
Then there exists a germ of holomorphic curve $(\gamma,o)$,
such that $\gamma\subset Y$, and $\gamma\not\subset A$.
\end{Lemma}

\begin{proof}
Note that $(Y,{o})$ is locally Stein.
Then using Cartan¡¯s Theorem $A$,
we obtain
the lemma.
\end{proof}

Now we recall the curve selection lemma stated as follows:

\begin{Lemma}
\label{l:curve}(see \cite{demailly2010})
Let $f$, $g_{1},\cdots,g_{s}\in\mathcal {O}_{n}$ be germs of holomorphic functions
vanishing at $0$.
Then we have $|f|\leq C|g|$ for some constant $C$
if and only if
for every germ of analytic curve $\gamma$ through $0$
there exists a constant $C_{\gamma}$
such that $|f\circ \gamma|\leq  C_{\gamma}|g\circ \gamma|$.
\end{Lemma}

In order to obtain some uniform properties of $\gamma$,
we need to consider the following Lemma
which was contained in the proof of Lemma \ref{l:curve} in \cite{demailly2010}.

\begin{Lemma}
\label{p:curve}(see \cite{demailly2010})
Let $f$, $g_{1},\cdots,g_{s}\in\mathcal {O}_{n}$ be germs of holomorphic functions
vanishing at $o$.
Assume that for any given neighborhood of $o$, $|f|\leq C|g|$ doesn't hold for any constant $C$,
where $g=(g_{1},\cdots,g_{s})$.
Then there exists a germ of analytic curve $\gamma$ through $o$,
satisfying
$\gamma\cap\{f=0\}=o$,
such that $\frac{g_{i}}{f}|_{\gamma}$ is holomorphic on $\gamma\setminus o$
with
$$\widetilde{\frac{g_{i}}{f}}|_{\gamma}(0)=0,$$
for any $i\in\{1,\cdots,s\}$,
where $\widetilde{\frac{g_{i}}{f}}$ is the holomorphic extension of $\frac{g_{i}}{f}$ from $\gamma\setminus o$ to $\gamma$.
\end{Lemma}

\begin{proof}
There exists $\Delta_{r}^{n}$,
such that $g_{1},\cdots,g_{s},f\in\mathcal{O}(\Delta_{r}^{n})$.
We define a germ of analytic set $(Y,o)\subset (\Delta^{n}_{r},o)$ by
$$g_{j}(z)=f(z)z_{n+j},\quad 1\leq j\leq s.$$

Let $p$ be a projection
$$p:\Delta^{n}_{r}\times\mathbb{C}^{s}\to\Delta^{n}_{r},$$
such that
$$p((z_{1},\cdots,z_{n}),(z_{n+1},\cdots,z_{n+s}))=(z_{1},\cdots,z_{n}).$$
Then $Y\cap p^{-1}(\Delta^{n}_{r}\setminus \{f=0\})$ is biholomorphic to $\Delta^{n}_{r}\setminus \{f=0\}$,
which is irreducible.
As every analytic variety has an irreducible decomposition,
then $Y$ contains an irreducible component $Y_{f}$ which contains $Y\cap p^{-1}(\Delta^{n}_{r}\setminus \{f=0\})$.

Since $Y_{f}$ is closed,
then
$$Y_{f}=\overline{Y\cap p^{-1}(\Delta^{n}_{r}\setminus \{f=0\})}.$$

By assumption,
for any given neighborhood of $o$,
$|f|\leq C|g|$ doesn't hold for any constant $C$,
then there exists a sequence of positive numbers $C_{\nu}$ which goes to $+\infty$ as $\nu\to\infty$,
and a sequence of points $\{z_{\nu}\}$ in $\Delta^{n}_{r}$ convergent to $o$ when $\nu\to\infty$,
such that
$|f(z_{\nu})|> C_{\nu}|g(z_{\nu})|$.

Then
$(z_{\nu},\frac{g_{j}(z_{\nu})}{f(z_{\nu})})$ converging to $0$ as $\nu$ tends to $+\infty$,
with $f(z_{\nu})\neq 0$.

As $(z_{\nu},\frac{g_{j}(z_{\nu})}{f(z_{\nu})})\in Y_{f}$,
then $Y_{f}$ contains $o$.

It follows from Lemma \ref{l:curve_exist}
that there exists a germ of analytic curve $(\gamma,\gamma_{n+j})\subset Y_{f}$ through $o$
satisfying
$\gamma\cap\{f=0\}=o$,
such that $\frac{g_{i}}{f}|_{\gamma}$ is holomorphic on $\gamma\setminus o$
for each $i\in\{1,\cdots,s\}$.

By the Riemann removable singularity theorems,
it follows that $\frac{g_{i}}{f}|_{\gamma\setminus o}$ can be extended to $\gamma$,
and
$$\widetilde{\frac{g_{i}}{f}}|_{\gamma}(0)=0,$$
for any $i\in\{1,\cdots,s\}$.
\end{proof}

\begin{Remark}
\label{r:curve927}
Let $g_{1},\cdots,g_{s}\in\mathcal {O}_{n}$ be germs of holomorphic functions
vanishing at $o$, and $f(o)\neq 0$.
Then there exists a germ of analytic curve $\gamma$ through $o$,
such that
$\gamma\cap\{f=0\}=\emptyset$,
and $\frac{g_{i}}{f}|_{\gamma}$ is holomorphic on $\gamma$
with $\frac{g_{i}}{f}|_{o}=0,$
for any $i\in\{1,\cdots,s\}$.
\end{Remark}

Let's recall a strong Noetherian property of coherent sheaves as follows:

\begin{Lemma}
\label{l:strong_noeth}(see \cite{demailly-book})
Let $\mathscr{F}$ be a coherent analytic sheaf on a complex manifold $M$,
and let $\mathscr{F}_{1}\subset\mathscr{F}_{2}\subset\cdots$ be an increasing of coherent subsheaves of $\mathscr{F}$.
Then the sequence $(\mathscr{F}_{k})$ is stationary on every compact subset of $M$.
\end{Lemma}

Let $\varphi$ be a negative
plurisubharmonic function on
$\Delta^{n}\subset\mathbb{C}^{n}$,
and $\{\psi_{j}\}_{j=1,2,\cdots}$ be a sequence of plurisubharmonic
functions on $\Delta^{n}$, which is increasingly convergent to
$\varphi$ on $\Delta^{n}$, when $j\to\infty$.

\begin{Remark}
\label{r:station}
By Lemma \ref{l:strong_noeth},
it is clear that
$\cup_{j=1}^{\infty}\mathcal{I}(\psi_{j})$ is a
coherent subsheaf of $\mathcal{I}(\varphi)$;
actually for any open $V_{1}\subset\subset M$,
there exists $j_{1}\in\{1,2,\cdots\}$,
such that $\cup_{j=1}^{\infty}\mathcal{I}(\psi_{j})|_{V_{1}}=\mathcal{I}(\psi_{j_{1}})$.
\end{Remark}

By Remark \ref{r:station},
we derive the following proposition about the generators of the coherent sheaf
$\cup_{j=1}^{\infty}\mathcal{I}(\psi_{j})$:

\begin{Proposition}
\label{r:curve}
Assume that $f\in\mathcal {O}_{n}$ is a holomorphic function on neighborhood $V_{0}$ of $o$,
which is not a germ of $(\cup_{j=1}^{\infty}\mathcal{I}(\psi_{j}))_{o}$.
Let $g_{1},\cdots,g_{s}\in\cup_{j=1}^{\infty}\mathcal{I}(\psi_{j})$ be germs of
holomorphic functions on some neighborhood $V_{1}\subset\subset V_{0}$ of $o$,
such that $g_{1},\cdots,g_{s}$ generate $\cup_{j=1}^{\infty}\mathcal{I}(\psi_{j})|_{V_1}$.
Then
there exists a germ of analytic curve $\gamma$ through $o$
satisfying $\gamma\cap\{f=0\}\subseteq\{o\}$,
such that
$\widetilde{\frac{g_{i}\circ\gamma}{f\circ\gamma}}$ is holomorphic on $\gamma$ for any $i\in\{1,\cdots,s\}$,
and
$$\widetilde{\frac{g_{i}\circ\gamma}{f\circ\gamma}}|_{0}=0,$$
where $\widetilde{\frac{g_{i}\circ\gamma}{f\circ\gamma}}$ is the holomorphic extension of
$\frac{g_{i}\circ\gamma}{f\circ\gamma}$ from $\gamma\setminus o$ to $\gamma$.
Moveover,
for any germ $g$ of $\mathcal{I}(\varphi+\varepsilon_{0}\varphi)_{0}$,
$\widetilde{\frac{g\circ\gamma}{f\circ\gamma}}$ is holomorphic on $\gamma\setminus o$,
and
$$\widetilde{\frac{g\circ\gamma}{f\circ\gamma}}|_{0}=0,$$
where $\widetilde{\frac{g\circ\gamma}{f\circ\gamma}}$ is the holomorphic extension of
$\frac{g\circ\gamma}{f\circ\gamma}$ from $\gamma\setminus o$ to $\gamma$.
\end{Proposition}

\begin{proof}
By Remark \ref{r:station}, it follows that
there exists $j_{1}\in\{1,2,\cdots\}$,
such that $g_{1},\cdots,g_{s}\in\mathcal{I}(\psi_{j_1})(V_{1})$.

As $f$ is not a germ of $(\cup_{j=1}^{\infty}\mathcal{I}(\psi_{j}))_{o}=\mathcal{I}(\psi_{j_1})_{o}$,
then for any neighborhood of $o$,
$|f|\leq C (\sum_{1\leq j\leq s}|g_{j}|^{2})^{1/2}$ doesn't hold for any constant $C$.

According to Lemma \ref{p:curve} and Remark \ref{r:curve927},
it follows that there exists a germ of analytic curve $\gamma$ through $o$
such that $\gamma\cap\{f=0\}\subseteq\{o\}$,
and
$\widetilde{\frac{g_{i}\circ\gamma}{f\circ\gamma}}$ is holomorphic on $\gamma$ for any $i$,
and
$$\widetilde{\frac{g_{i}\circ\gamma}{f\circ\gamma}}|_{0}=0,$$
where $\widetilde{\frac{g_{i}\circ\gamma}{f\circ\gamma}}$ is the holomorphic extension of
$\frac{g_{i}\circ\gamma}{f\circ\gamma}$ from $\gamma\setminus o$ to $\gamma$.
\end{proof}

\begin{Remark}\label{r:131011}
Let $(f,o)$ be a germ of holomorphic function on $\Omega\ni o$, such
that $(f,o)\not\in\mathcal{I}(\varphi)_{o}$. Assume
$\{F_{l}\}_{l=1,2\cdots}$ is a sequence of holomorphic functions on
$\Omega$, such that $F_{l}$ is uniformly convergent to a holomorphic
function $F$ on any compact subset, and
$(F_{l}-f,o)\in\mathcal{I}(\varphi)_{o}$. Then $F\not\equiv0$.
\end{Remark}

\begin{proof}
As $(f,o)\not\in\mathcal{I}(\varphi)_{o}$,
let $\psi_{j}=\varphi$ for any $j\in\{1,2,\cdots\}$,
it follows from Proposition \ref{r:curve}
that there
exists a germ of analytic curve $\gamma$ through $o$,
such that for any germ $F_{l}-f$ of $\mathcal{I}(\varphi)_{o}$,
$\widetilde{\frac{(F_{l}-f)\circ\gamma}{f\circ\gamma}}$ is holomorphic on $\gamma\setminus o$,
and
$$\widetilde{\frac{(F_{l}-f)\circ\gamma}{f\circ\gamma}}|_{0}=0,$$
where $\widetilde{\frac{(F_{l}-f)\circ\gamma}{f\circ\gamma}}$ is the holomorphic extension of
$\frac{(F_{l}-f)\circ\gamma}{f\circ\gamma}$ from $\gamma\setminus o$ to $\gamma$.

Therefore
$$\widetilde{\frac{F_{l}\circ\gamma}{f\circ\gamma}}|_{0}=\widetilde{\frac{f\circ\gamma}{f\circ\gamma}}|_{0}=1.$$

As $F_{l}$ is uniformly convergent to $F$,
it follows that $F|_{\gamma}\not\equiv 0$,
which implies $F\not\equiv 0$.
\end{proof}

\subsection{$L^{2}$ estimates for some $\bar\partial$ equations}
$\\$

For the sake of completeness, we recall some lemmas on $L^{2}$ estimates for some $\bar\partial$ equations, and $\bar\partial^*$ means the Hilbert adjoint operator of
$\bar\partial$.

\begin{Lemma}\label{l:lem3}(see \cite{siu96}, see alse \cite{berndtsson})
Let $\Omega\subset\subset \mathbb{C}^{n}$ be a domain with $C^\infty$ boundary $b\Omega$,
$\Phi\in C^{\infty}(\overline \Omega)$,
 Let $\rho$ be a $C^{\infty}$ defining function for $\Omega$
such that $|d\rho|=1$ on $b\Omega$.
Let $\eta$ be a smooth function on $\overline{\Omega}$. For any $(0,1)$-form
$\alpha=\sum_{j=1}^{n}\alpha_{\bar{j}}d\bar z^{j}\in Dom_\Omega(\bar{\partial}^*)\cap C^\infty_{(0,1)}(\overline\Omega)$,
\begin{equation}
\label{guan1}
\begin{split}
&\int_{\Omega}\eta|\bar{\partial}^{*}_{\Phi}\alpha|^{2}e^{-\Phi}d\lambda_{n}
+\int_{\Omega}\eta|\bar{\partial}\alpha|^{2}e^{-\Phi}d\lambda_{n}
=\sum_{i,j=1}^{n}\int_{\Omega}
\eta |\overline\partial_{j}\alpha_{\bar{j}}|^{2}d\lambda_{n}
\\&+\sum_{i,j=1}^{n}
\int_{b\Omega}\eta(\partial_i\bar\partial_{j}\rho)\alpha_{\bar{i}}\overline{{\alpha}_{\bar{j}}}e^{-\Phi}dS
+\sum_{i,j=1}^{n}\int_{\Omega}
\eta(\partial_i\bar \partial_{j}\Phi)\alpha_{\bar{i}}\overline{{\alpha}_{\bar{j}}}e^{-\Phi}d\lambda_{n}
\\&+\sum_{i,j=1}^{n}\int_{\Omega}
-(\partial_i\bar \partial_{j}\eta)\alpha_{\bar{i}}\overline{{\alpha}_{\bar j}}e^{-\Phi}d\lambda_{n}
+2\mathrm{Re}(\bar\partial^*_\Phi\alpha,\alpha\llcorner(\bar\partial\eta)^\sharp )_{\Omega,\Phi},
\end{split}
\end{equation}
where $d\lambda_{n}$ is the Lebesgue measure on $\mathbb{C}^{n}$, and
$\alpha\llcorner(\bar\partial\eta)^\sharp=\sum_{j}\alpha_{\bar{j}}\partial_{j}\eta$.
\end{Lemma}

The symbols and notations can be referred to \cite{guan-zhou-zhu10}. See also \cite{siu96}, \cite{siu00}, or \cite{Straube}.

\begin{Lemma}\label{l:lem7}(see \cite{berndtsson}, see also \cite{guan-zhou-zhu10})
Let $\Omega\subset\subset \mathbb{C}^{n}$ be a strictly pseudoconvex domain with $C^\infty$ boundary $b\Omega$ and $\Phi\in C^\infty(\overline{\Omega})$. Let $\lambda$ be a $\bar\partial$ closed smooth form of bidgree $(n,1)$ on $\overline{\Omega}$. Assume the inequality
$$|(\lambda,\alpha)_{\Omega,\Phi}|^{2}\leq C\int_{ \Omega}|\bar{\partial}^{*}_{\Phi}\alpha|^{2}\frac{e^{-\Phi}}{\mu}d\lambda_{n}<\infty,$$
where $\frac{1}{\mu}$ is an integrable positive function on $\Omega$ and
$C$ is a constant, holds for all $(n,1)$-form $\alpha\in Dom_{\Omega}(\bar\partial^*)\cap Ker(\bar\partial)\cap C^\infty_{(n,1)}(\overline \Omega)$. Then there is a solution $u$ to the
equation $\bar{\partial}u=\lambda$ such that
$$\int_{\Omega}|u|^{2}\mu e^{-\Phi}d\lambda_{n}\leq C.$$
\end{Lemma}

\section{Some propositions on multiplier ideal sheaves}

\subsection{A proposition used in the proof of  Conjecture D-K}
$\\$

We prove Theorem \ref{t:GZ_JM} and \ref{t:GZ_open1026} by the following proposition:

\begin{Proposition} \label{p:GZ_JM}
Let $D_{v}$ be a strongly pseudoconvex domain relatively compact in
$\Delta^{n}$ containing $o$. Let $F$ be a holomorphic function on
$\Delta^{n}$. Let $\varphi$ be a negative plurisubharmonic function
on $\Delta^{n}$, such that $\varphi(o)=-\infty$. Then there exists a
holomorphic function $F_{v,t_{0}}$ on $D_{v}$, such that,
$$(F_{v,t_{0}}-F,o)\in\mathcal{I}(\varphi)_{o}$$
and
\begin{equation}
\label{equ:3.4}
\begin{split}
&\int_{ D_v}|F_{v,t_0}-(1-b_{t_0}(\varphi))F|^{2}d\lambda_{n}
\\\leq&\int_{D_v}(\mathbb{I}_{\{-t_{0}-1< t<-t_{0}\}}\circ\varphi)|F|^{2}e^{-\varphi}d\lambda_{n},
\end{split}
\end{equation}
where
$b_{t_{0}}(t)=\int_{-\infty}^{t}\mathbb{I}_{\{-t_{0}-1< s<-t_{0}\}}ds$.
\end{Proposition}

When $\varphi$ is a polar function (see \cite{ohsawa5}, see also \cite{guan-zhou13a,guan-zhou13p,guan-zhou13aa,guan-zhou13ap}),
various versions of Proposition \ref{p:GZ_JM} were used to prove the main results in
\cite{ohsawa5}, \cite{guan-zhou13a}, \cite{guan-zhou13p}, \cite{guan-zhou13aa}, \cite{guan-zhou13ap}, etc.

\subsection{A proposition used in the proof of Conjecture J-M}
$\\$

Let
$$\varphi:=2\max\{\psi,\log|F|^{2}\},$$
and
$$\Psi:=\min\{\psi-\log|F|^{2},0\}-1.$$
Then $\Psi+\varphi$ and $2\Psi+\varphi$ are both plurisubharmonic functions on $\Delta^{n}$.

Note that
$$e^{-1}\mathbb{I}_{\{\psi\leq\log|F|^{2}\}}e^{-\Psi}e^{-\Psi}\leq
|F|^{2}e^{-\psi}\leq e^{-1}e^{-\Psi},$$
and
$$e^{-1}e^{-\Psi}-e^{-1}\mathbb{I}_{\{\psi\leq\log|F|^{2}\}}e^{-\Psi}
=e^{-1}\mathbb{I}_{\{\psi>\log|F|^{2}\}}e^{-\Psi}=e^{-1}\mathbb{I}_{\{\psi>\log|F|^{2}\}},$$
then we have the following three statements are equivalent:

1). $|F|^{2}e^{-\psi}$ is not locally integrable near $o$£»

2). $e^{-\Psi}$ is not locally integrable near $o$£»

3). $\mathbb{I}_{\{\psi\leq\log|F|^{2}\}}e^{-\Psi}$ is not locally integrable near $o$.

Note that
$$|F|^{4}e^{-\varphi}|_{\{\psi\leq\log|F|^{2}\}}=e^{-\max\{\psi-\log|F|^{2},0\}}|_{\{\psi\leq\log|F|^{2}\}}=1,$$
then we have:

\begin{Remark}
\label{r:20131210}
The following two statements are equivalent:

1). $|F|^{2}e^{-\psi}$ is not locally integrable near $o$;

2). $|F|^{4}e^{-\Psi-\varphi}$ is not locally integrable near $o$.
\end{Remark}

We prove Theorem \ref{t:GZ_JM201312} by the following proposition:

\begin{Proposition} \label{p:GZ_JM201312}
Let $D_{v}$ be a strongly pseudoconvex domain relatively compact in $\Delta^{n}$ containing $o$.

Then there exists  a holomorphic function $F_{v,t_{0}}$ on $D_{v}$,
satisfying:
\begin{equation}
\label{equ:20131210a}
(F_{v,t_{0}}-F^{2},o)\in\mathcal{I}(\varphi+\Psi)_{o}
\end{equation}
and
\begin{equation}
\label{equ:3.4}
\begin{split}
&\int_{ D_v}|F_{v,t_0}-(1-b_{t_0}(\Psi))F^2|^{2}e^{-\varphi}d\lambda_{n}
\\\leq&2\int_{D_v}(\frac{1}{B_{0}}\mathbb{I}_{\{-t_{0}-B_{0}< t<-t_{0}\}}\circ\Psi)|F^2|^{2}e^{-\varphi}e^{-\Psi}d\lambda_{n},
\end{split}
\end{equation}
where
$$b_{t_{0}}:=\int_{-\infty}^{t}\frac{1}{B_{0}}\mathbb{I}_{(-t_{0}-B_{0},-t_{0})}ds,$$
and $t_{0}\geq0.$
\end{Proposition}

\subsection{A smooth form of Proposition \ref{p:GZ_JM201312}}
$\\$

Let $\{v_{t_0,\varepsilon}\}_{t_{0}\in\mathbb{R},\varepsilon\in(0,\frac{1}{8}B_{0})}$ be a family of smooth increasing convex functions on $\mathbb{R}$,
which are continuous functions on $\mathbb{R}\cup\{-\infty\}$, such that:

 $1).$ $v_{t_{0},\varepsilon}(t)=t$ for $t\geq-t_{0}-\varepsilon$, $v_{t_{0},\varepsilon}(t)=constant$ for $t<-t_{0}-B_{0}+\varepsilon$;

 $2).$ $v''_{t_0,\varepsilon}(t)$ are pointwise convergent to $\frac{1}{B_{0}}\mathbb{I}_{(-t_{0}-B_{0},-t_{0})}$, when $\varepsilon\to 0$, and $0\leq v''_{t_0,\varepsilon}(t)\leq 2$ for any $t\in \mathbb{R}$;

 $3).$ $v'_{t_0,\varepsilon}(t)$ are pointwise convergent to $b_{t_{0}}(t)=\int_{-\infty}^{t}\frac{1}{B_{0}}\mathbb{I}_{(-t_{0}-B_{0},-t_{0})}ds$ ($b_{t_{0}}$ is also a continuous function on $\mathbb{R}\cup\{-\infty\}$), when $\varepsilon\to 0$, and $0\leq v'_{t_0,\varepsilon}(t)\leq1$ for any $t\in \mathbb{R}$.

One can construct the family $\{v_{t_0,\varepsilon}\}_{t_{0}\in\mathbb{R},\varepsilon\in(0,\frac{1}{8})}$ by the setting
\begin{equation}
\label{equ:20140101}
\begin{split}
v_{t_0,\varepsilon}(t):=&\int_{-\infty}^{t}(\int_{-\infty}^{t_{1}}(\frac{1}{1-4\varepsilon}
\frac{1}{B_{0}}\mathbb{I}_{(-t_{0}-B_{0}+2\varepsilon,-t_{0}-2\varepsilon)}*\rho_{\frac{1}{4}\varepsilon})(s)ds)dt_{1}
\\&-\int_{-\infty}^{0}(\int_{-\infty}^{t_{1}}(\frac{1}{1-4\varepsilon}\frac{1}{B_{0}}\mathbb{I}_{(-t_{0}-B_{0}+2\varepsilon,
-t_{0}-2\varepsilon)}*\rho_{\frac{1}{4}\varepsilon})(s)ds)dt_{1},
\end{split}
\end{equation}
where $\rho_{\frac{1}{4}\varepsilon}$ is the kernel of convolution satisfying $supp(\rho_{\frac{1}{4}\varepsilon})\subset (-\frac{1}{4}\varepsilon,\frac{1}{4}\varepsilon)$.
Then it follows that
$$v''_{t_0,\varepsilon}(t)=\frac{1}{1-4\varepsilon}\frac{1}{B_{0}}\mathbb{I}_{(-t_{0}-B_{0}+2\varepsilon,-t_{0}-2\varepsilon)}*\rho_{\frac{1}{4}\varepsilon}(t),$$
and
$$v'_{t_0,\varepsilon}(t)=\int_{-\infty}^{t}(\frac{1}{1-4\varepsilon}\frac{1}{B_{0}}\mathbb{I}_{(-t_{0}-B_{0}+2\varepsilon,-t_{0}-2\varepsilon)}
*\rho_{\frac{1}{4}\varepsilon})(s)ds.$$

We prove Proposition \ref{p:GZ_JM201312} by the following proposition:

\begin{Proposition} \label{p:GZ_JM_smooth}
Let $D_{v}$ be a strongly pseudoconvex domain relatively compact in $\Delta^{n}$ containing $o$.
Then there exists  a holomorphic function $F_{v,t_{0},\varepsilon}$ on $D_{v}$,
satisfying:
\begin{equation}
\label{equ:20131210a}
(F_{v,t_{0},\varepsilon}-F^{2},o)\in\mathcal{I}(\varphi+\Psi)_{o},
\end{equation}
and
\begin{equation}
\label{equ:JM3.4}
\begin{split}
&\int_{ D_v}|F_{v,t_{0},\varepsilon}-(1-v_{t_{0},\varepsilon}'(\Psi))F^2|^{2}e^{-\varphi}d\lambda_{n}
\\\leq&2\int_{D_v}(\frac{1}{B_{0}}v_{t_{0},\varepsilon}''(\Psi))|F^2|^{2}e^{-\varphi}e^{-\Psi}d\lambda_{n},
\end{split}
\end{equation}
where $t_{0}\geq0.$
\end{Proposition}

\section{Proofs of the propositions}

\subsection{Proof of Proposition \ref{p:GZ_JM}}
$\\$

For the sake of completeness, let's recall some step in our proof in
\cite{guan-zhou13p} (see also \cite{guan-zhou13ap}) with some
modifications in order to prove Proposition \ref{p:GZ_JM}.

Let $\{v_{t_0,\varepsilon}\}_{t_{0}\in\mathbb{R},\varepsilon\in(0,\frac{1}{8})}$ be a family of smooth increasing convex functions on $\mathbb{R}$,
which are continuous functions on $\mathbb{R}\cup\{-\infty\}$, such that:

 $1).$ $v_{t_{0},\varepsilon}(t)=t$ for $t\geq-t_{0}-\varepsilon$, $v_{t_{0},\varepsilon}(t)=constant$ for $t<-t_{0}-1+\varepsilon$;

 $2).$ $v''_{t_0,\varepsilon}(t)$ are pointwise convergent to $\mathbb{I}_{(-t_{0}-1,-t_{0})}$, when $\varepsilon\to 0$, and $0\leq v''_{t_0,\varepsilon}(t)\leq 2$ for any $t\in \mathbb{R}$;

 $3).$ $v'_{t_0,\varepsilon}(t)$ are pointwise convergent to $b_{t_{0}}(t)=\int_{-\infty}^{t}\mathbb{I}_{(-t_{0}-1,-t_{0})}ds$ ($b_{t_{0}}$ is also a continuous function on $\mathbb{R}\cup\{-\infty\}$), when $\varepsilon\to 0$, and $0\leq v'_{t_0,\varepsilon}(t)\leq1$ for any $t\in \mathbb{R}$.

One can construct the family $\{v_{t_0,\varepsilon}\}_{t_{0}\in\mathbb{R},\varepsilon\in(0,\frac{1}{8})}$ as equality \ref{equ:20140101} by taking $B_{0}=1$.
Then it follows that
$$v''_{t_0,\varepsilon}(t)=\frac{1}{1-4\varepsilon}\mathbb{I}_{(-t_{0}-1+2\varepsilon,-t_{0}-2\varepsilon)}*\rho_{\frac{1}{4}\varepsilon}(t),$$
and
$$v'_{t_0,\varepsilon}(t)=\int_{-\infty}^{t}(\frac{1}{1-4\varepsilon}\mathbb{I}_{(-t_{0}-1+2\varepsilon,-t_{0}-2\varepsilon)}
*\rho_{\frac{1}{4}\varepsilon})(s)ds.$$

As $D_{v}\subset\subset\Delta^{n}\subset \mathbb{C}^{n}$,
then there exist negative smooth plurisubharmonic functions $\{\varphi_{m}\}_{m=1,2,\cdots}$ on a neighborhood of $\overline{D}_{v}$,
such that the sequence $\{\varphi_{m}\}_{m=1,2,\cdots}$ is decreasingly convergent to $\varphi$
on a smaller neighborhood of $\overline{D}_{v}$,
when $m\to+\infty$.

Let $\eta=s(-v_{t_{0},\varepsilon}\circ\varphi_{m})$ and $\phi=u(-v_{t_{0},\varepsilon}\circ\varphi_{m})$,
where $s\in C^{\infty}((0,+\infty))$ satisfies $s\geq0$, and
$u\in C^{\infty}((0,+\infty))$, satisfies $\lim_{t\to+\infty}u(t)=0$, such that $u''s-s''>0$, and $s'-u's=1$.

Let $\Phi=\varphi_{m}+\phi$.

Now let $\alpha=\sum^{n}_{j=1}\alpha_{j}d\bar z^{j}\in Dom_{D_v}
(\bar\partial^*)\cap Ker(\bar\partial)\cap C^\infty_{(0,1)}(\overline {D_v})$.
By Cauchy-Schwarz inequality, it follows that
\begin{equation}
\label{equ:20131130a}
\begin{split}
2\mathrm{Re}(\bar\partial^*_\Phi\alpha,\alpha\llcorner(\bar\partial\eta)^\sharp )_{\Omega,\Phi}
\geq
&-\int_{D_v}g^{-1}|\bar{\partial}^{*}_{\Phi}\alpha|^{2}e^{-\Phi}d\lambda_{n}
\\&+
\sum_{j,k=1}^{n}\int_{D_v}
(-g(\partial_{j} \eta)\bar\partial_{k} \eta )\alpha_{\bar{j} }\overline{{\alpha}_{\bar {k}}}e^{-\Phi}d\lambda_{n}.
\end{split}
\end{equation}

Using Lemma \ref{l:lem3} and inequality \ref{equ:20131130a},
since $s\geq0$ and $\varphi_{m}$ is a plurisubharmonic function on $\overline{D}_{v}$,
we get

\begin{equation}
\label{equ:4.1}
\begin{split}
&\int_{D_v}(\eta+g^{-1})|\bar{\partial}^{*}_{\Phi}\alpha|^{2}e^{-\Phi}d\lambda_{n}
\\\geq&\sum_{j,k=1}^{n}\int_{D_v}
(-\partial_{j}\bar{\partial}_{k}\eta+\eta\partial_{j}\bar{\partial}_{k}\Phi-g(\partial_{j} \eta)\bar\partial_{k} \eta )\alpha_{\bar{j} }\overline{{\alpha}_{\bar{k}}}e^{-\Phi}d\lambda_{n}
\\\geq&\sum_{j,k=1}^{n}\int_{D_v}
(-\partial_{j}\bar{\partial}_{k}\eta+\eta\partial_{j}\bar{\partial}_{k}\phi-g(\partial_{j} \eta)\bar\partial_{k} \eta )\alpha_{\bar{j} }\overline{{\alpha}_{\bar {k}}}e^{-\Phi}d\lambda_{n},
\end{split}
\end{equation}
where $g$ is a positive continuous function on $D_{v}$.
We need some calculations to determine $g$.

We have

\begin{equation}
\label{}
\begin{split}
&\partial_{j}\bar{\partial}_{k}\eta=-s'(-v_{t_0,\varepsilon}\circ \varphi_{m})\partial_{j}\bar{\partial}_{k}(v_{t_0,\varepsilon}\circ \varphi_{m})
\\+&s''(-v_{t_0,\varepsilon}\circ \varphi_{m})\partial_{j}(v_{t_0,\varepsilon}\circ \varphi_{m})
\bar{\partial}_{k}(v_{t_0,\varepsilon}\circ\varphi_{m}),
\end{split}
\end{equation}

and
\begin{equation}
\label{}
\begin{split}
&\partial_{j}\bar{\partial}_{k}\phi=-u'(-v_{t_0,\varepsilon}\circ \varphi_{m})\partial_{j}\bar{\partial}_{k}v_{t_0,\varepsilon}\circ \varphi_{m}
\\+&
u''(-v_{t_0,\varepsilon}\circ \varphi_{m})\partial_{j}(v_{t_0,\varepsilon}\circ \varphi_{m})\bar{\partial}_{k}(v_{t_0,\varepsilon}\circ \varphi_{m})
\end{split}
\end{equation}

for any $j,k$ ($1\leq j,k\leq n$).

We have

\begin{equation}
\label{}
\begin{split}
&\sum_{1\leq j,k\leq n}(-\partial_{j}\bar{\partial}_{k}\eta+\eta\partial_{j}\bar{\partial}_{k}\phi-g(\partial_{j} \eta)
\bar\partial_{k} \eta )\alpha_{\bar{j} }\overline{{\alpha}_{\bar{k}}}
\\=&(s'-su')\sum_{1\leq j,k\leq n}\partial_{j}\bar{\partial}_{k}(v_{t_0,\varepsilon}\circ \varphi_{m})\alpha_{\bar{j} }\overline{{\alpha}_{\bar{k}}}
\\+&((u''s-s'')-gs'^{2})\sum_{1\leq j,k\leq n}\partial_{j}
(-v_{t_0,\varepsilon}\circ \varphi_{m})\bar{\partial}_{k}(-v_{t_0,\varepsilon}\circ\varphi_{m})\alpha_{\bar{j} }\overline{{\alpha}_{\bar{k}}}
\\=&(s'-su')\sum_{1\leq j,k\leq n}((v'_{t_0,\varepsilon}\circ\varphi_{m})\partial_{j}\bar{\partial}_{k}\varphi_{m}+(v''_{t_0,\varepsilon}\circ \varphi_{m})\partial_{j}(\varphi_{m})\bar{\partial}_{k}(\varphi_{m}))\alpha_{\bar{j} }\overline{{\alpha}_{\bar{k}}}
\\+&((u''s-s'')-gs'^{2})\sum_{1\leq j,k\leq n}\partial_{j}
(-v_{t_0,\varepsilon}\circ \varphi_{m})\bar{\partial}_{k}(-v_{t_0,\varepsilon}\circ \varphi_{m})\alpha_{\bar{j} }\overline{{\alpha}_{\bar{k}}}.
\end{split}
\end{equation}
We omit composite item $(-v_{t_0,\varepsilon}\circ \varphi_{m})$ after $s'-su'$ and $(u''s-s'')-gs'^{2}$ in the above equalities.

Let $g=\frac{u''s-s''}{s'^{2}}\circ(-v_{t_0,\varepsilon}\circ \varphi_{m})$.
It follows that $\eta+g^{-1}=(s+\frac{s'^{2}}{u''s-s''})\circ(-v_{t_0,\varepsilon}\circ \varphi_{m})$.

Because of $v'_{t_0,\varepsilon}\geq 0$  and $s'-su'=1$, using inequalities \ref{equ:4.1}, we have
\begin{equation}
\label{equ:3.1}
\begin{split}
\int_{D_v}(\eta+g^{-1})|\bar{\partial}^{*}_{\Phi}\alpha|^{2}e^{-\Phi}d\lambda_{n}
\geq\int_{D_v}(v''_{t_0,\varepsilon}\circ{\varphi_{m}})
\big|\alpha\llcorner(\bar \partial \varphi_{m})^\sharp\big|^2e^{-\Phi}d\lambda_{n}.
\end{split}
\end{equation}

Let $\lambda=\bar{\partial}[(1-v'_{t_0,\varepsilon}(\varphi_{m})){F}]$.
By the definition of contraction, Cauchy-Schwarz inequality and inequality \ref{equ:3.1},
it follows that
\begin{equation}
\begin{split}
&|(\lambda,\alpha)_{D_v,\Phi}|^{2}=|((v''_{t_0,\varepsilon}\circ{\varphi_{m}})\bar\partial\varphi_{m} F,\alpha)_{D_v,\Phi}|^{2}
\\=&|((v''_{t_0,\varepsilon}\circ{\varphi_{m}})F,\alpha\llcorner(\bar\partial\varphi_{m})^\sharp\big)_{D_v,\Phi}|^{2}
\\\leq&\int_{D_v}
(v''_{t_0,\varepsilon}\circ{\varphi_{m}})| F|^2e^{-\Phi}d\lambda_{n}\int_{D_v}(v''_{t_0,\varepsilon}\circ{\varphi_{m}})
\big|\alpha\llcorner(\bar\partial\varphi_{m})^\sharp\big|^2e^{-\Phi}d\lambda_{n}
\\\leq&
(\int_{D_v}
(v''_{t_0,\varepsilon}\circ{\varphi_{m}})| F|^2e^{-\Phi}d\lambda_{n})
(\int_{ D_v}(\eta+g^{-1})|\bar{\partial}^{*}_{\Phi}\alpha|^{2}e^{-\Phi}d\lambda_{n}).
\end{split}
\end{equation}

Let $\mu:=(\eta+g^{-1})^{-1}$. Using Lemma \ref{l:lem7},
we have locally $L^{1}$ function $u_{v,t_0,m,\varepsilon}$ on $D_{v}$ such that $\bar{\partial}u_{v,t_0,m,\varepsilon}=\lambda$,
and
\begin{equation}
 \label{equ:3.2}
 \begin{split}
 &\int_{ D_v}|u_{v,t_0,m,\varepsilon}|^{2}(\eta+g^{-1})^{-1} e^{-\Phi}d\lambda_{n}
  \leq\int_{D_v}(v''_{t_0,\varepsilon}\circ{\varphi_{m}})| F|^2e^{-\Phi}d\lambda_{n}.
  \end{split}
\end{equation}

Let $\mu_{1}=e^{v_{t_0,\varepsilon}\circ\varphi_{m}}$, $\tilde{\mu}=\mu_{1}e^{\phi}$.
Assume that we can choose $\eta$ and $\phi$ such that $\tilde{\mu}\leq \mathbf{C}(\eta+g^{-1})^{-1}=\mu$, where $\mathbf{C}=1$.

Note that $v_{t_0,\varepsilon}(\varphi_{m})\geq\varphi_{m}$.
Then it follows that
\begin{equation}
\label{equ:3.8}
\begin{split}
\int_{ D_v}|u_{v,t_0,m,\varepsilon}|^{2} d\lambda_{n}\leq\int_{ D_v}|u_{v,t_0,m,\varepsilon}|^{2}\mu_{1}e^{\phi} e^{-\varphi_{m}-\phi}d\lambda_{n}
=\int_{ D_v}|u_{v,t_0,m,\varepsilon}|^{2}\tilde{\mu}e^{-\Phi}d\lambda_{n}.
\end{split}
\end{equation}

Using inequalities \ref{equ:3.2} and \ref{equ:3.8},
we obtain that
$$\int_{ D_v}|u_{v,t_0,m,\varepsilon}|^{2} d\lambda_{n}\leq\mathbf{C}\int_{D_v}
(v''_{t_0,\varepsilon}\varphi_{m})| F|^2e^{-\Phi}d\lambda_{n},$$
under the assumption $\tilde{\mu}\leq\mathbf{C} (\eta+g^{-1})^{-1}$.

As $-v_{t_0,\varepsilon}\circ{\varphi_{m}}(\overline{D_{v}})\subset\subset(0,t_{0}+1)$
and $\{\varphi_{m}\}_{m=1,2,\cdots}$ is decreasing,
then it is clear that
\begin{equation}
\label{equ:20131130b}
-v_{t_0,\varepsilon}\circ{\varphi_{m}}(\overline{D_{v}})\subset\subset K_{t_{0}}\subset\subset(0,t_{0}+1)
\end{equation}
where $K_{t_{0}}$ is independent of $m$ and $\varepsilon\in (0,\frac{1}{8})$.
As $u$ is positive and smooth on $(0,+\infty)$,
it follows that $\phi$ is uniformly bounded on $\overline{D}_{v}$ independent of $m$.

As $Supp(v''_{t_{0},\varepsilon})\subset\subset(-t_{0}-1,-t_{0})$,
then it is clear that $(v''_{t_0,\varepsilon}\circ{\varphi_{m}})| F|^2e^{-\varphi_{m}}$ are uniformly bounded on $\overline{D}_{v}$ independent of $m$.
Therefore $\int_{D_v}
(v''_{t_0,\varepsilon}\circ{\varphi_{m}})| F|^2e^{-\Phi}d\lambda_{n}$ are uniformly bounded independent of $m$,
for any given $v$, $t_0$, $\varepsilon$.

By weakly compactness of the unit ball of $L^{2}(D_{v})$ and dominated convergence theorem, when $m\to+\infty$,
it follows that the weak limit of some weakly convergent subsequence of
$\{u_{v,t_{0},m,\varepsilon}\}_{m}$ gives an $(n,0)$-form $u_{v,t_0,\varepsilon}$ on $D_v$ satisfying
\begin{equation}
\label{equ:3.3}\int_{ D_v}|u_{v,t_0,\varepsilon}|^{2}d\lambda_{n}\leq\frac{\mathbf{C}}{e^{A_{t_0}}}\int_{D_v}
(v''_{t_0,\varepsilon}\circ{\varphi})| F|^2e^{-\varphi}d\lambda_{n},
\end{equation}
where $A_{t_0}:=\inf_{t\geq t_0}\{u(t)\}$.

As $\varphi_{m}$ is decreasingly convergent to $\varphi$ on $\Delta^{n}$,
and $\varphi(o)=-\infty$,
then for any given $t_{0}$ there exists $m_{0}$ and a neighbourhood $U_{0}$ of $o\in D_{v}$ on $\Delta^{n}$,
such that for any $m\geq m_{0}$ and $\varepsilon<1$,
$v''_{t_0,\varepsilon}\circ\varphi_{m}|_{U_{0}}=0$.
It follows that
$$\bar\partial u_{v,t_0,m,\varepsilon}|_{U_0}=\lambda|_{U_0}=\bar{\partial}[(1-v'_{t_0,\varepsilon}(\varphi_{m})){F}]|_{U_0}=-(v''_{t_0,\varepsilon}\circ\varphi_{m})F\bar{\partial}\varphi_{m}|_{U_{0}}=0.$$
That is to say $u_{v,t_0,m,\varepsilon}|_{U_0}$ are all holomorphic.
Therefore $u_{v,t_0,\varepsilon}|_{U_0}$ is holomorphic.

Recall that the integrals $\int_{ D_v}|u_{v,t_{0},m,\varepsilon}|^{2} d\lambda_{n}$ have a uniform bound independent of $m$,
then we can choose a subsequence with respect to $m$ from the chosen weakly convergent subsequence of $u_{v,t_{0},m,\varepsilon}$,
such that the subsequence is uniformly convergent on any compact subset of $U_0$,
and we still denote the subsequence by $u_{v,t_{0},m,\varepsilon}$ without ambiguity.

By above arguments,
it follows that the right hand side of inequality \ref{equ:3.2} are uniformly bounded independent of $m$ and $\varepsilon\in(0,\frac{1}{8})$.
By inequality \ref{equ:3.2},
it follows that $\int_{ D_v}|u_{v,t_0,m,\varepsilon}|^{2}(\eta+g^{-1})^{-1} e^{-\phi-\varphi_{m}}d\lambda_{n}$
are uniformly bounded independent of $m$ and $\varepsilon\in(0,\frac{1}{8})$.

Using inequality \ref{equ:20131130b},
we obtain that
$$(\eta+g^{-1})^{-1}=(s(-v_{t_{0},\varepsilon}\circ\varphi_{m})+\frac{s'^{2}}{u''s-s''}\circ(-v_{t_0,\varepsilon}\circ\varphi_{m}))^{-1}$$
and
$e^{-\phi}=e^{-u(-v_{t_{0},\varepsilon}\circ\varphi_{m})}$ have positive uniform lower bounds independent of $m$ and $\varepsilon\in(0,\frac{1}{8})$.
Then the integrals
$$\int_{ K_0}|u_{v,t_0,m,\varepsilon}|^{2}e^{-\varphi_{m}}d\lambda_{n}$$
have a uniform upper bound independent of $m$ and $\varepsilon\in(0,\frac{1}{8})$
for any given compact set $K_{0}\subset\subset U_0\cap D_v$,
and
$$\bar\partial u_{v,t_0,m,\varepsilon}|_{U_0}=0.$$

As
$\varphi_{m'}\leq\varphi_{m}$ where $m'\geq m$,
it follows that
$$|u_{v,t_0,m',\varepsilon}|^{2}e^{-\varphi_{m}}\leq|u_{v,t_0,m',\varepsilon}|^{2}e^{-\varphi_{m'}}.$$
Then for any given compact set $K_{0}\subset\subset U_0\cap D_v$,
$\int_{ K_0}|u_{v,t_0,m',\varepsilon}|^{2}e^{-\varphi_{m}}d\lambda_{n}$
have a uniform bound independent of $m$ and $m'$.
It is clear that
for any given compact set $K_{0}\subset\subset U_0\cap D_v$,
the integrals $\int_{ K_0}|u_{v,t_0,\varepsilon}|^{2}e^{-\varphi_{m}}d\lambda_{n}$
have a uniform bound independent of $m$ and $\varepsilon\in(0,\frac{1}{8})$.
Therefore the integrals $\int_{ K_0}|u_{v,t_0,\varepsilon}|^{2}e^{-\varphi}d\lambda_{n}$
have a uniform bound independent of $\varepsilon\in(0,\frac{1}{8})$,
for any given compact set $K_{0}\subset\subset U_0\cap D_v$ containing $o$.

In summary,
we have $|u_{v,t_0,\varepsilon}|^{2}e^{-\varphi}$ is integrable near $o$,
and $\bar\partial u_{v,t_0,\varepsilon}=0$ near $o$.
That is to say
$$(u_{v,t_0,\varepsilon},o)\in\mathcal{I}(\varphi)_{o}.$$

Let $F_{v,t_0,\varepsilon}:=(1-v'_{t_0,\varepsilon}\circ\varphi)F-u_{v,t_0,\varepsilon}$.
By inequality \ref{equ:3.3} and $(u_{v,t_0,\varepsilon},o)\in\mathcal{I}(\varphi)_{o}$,
it follows that
$F_{v,t_0,\varepsilon}$ is a holomorphic function on $D_{v}$ satisfying $(F_{v,t_0,\varepsilon}-F,o)\in\mathcal{I}(\varphi)_{o}$ and
\begin{equation}
\label{equ:3.5}
\begin{split}
&\int_{ D_v}|F_{v,t_0,\varepsilon}-(1-v'_{t_0,\varepsilon}\circ\varphi)F|^{2}d\lambda_{n}
\\\leq&\frac{\mathbf{C}}{e^{A_{t_0}}}\int_{D_v}(v''_{t_0,\varepsilon}\circ\varphi)|F|^{2}e^{-\varphi}d\lambda_{n}.
\end{split}
\end{equation}

Given $t_0$ and $D_{v}$,
it is clear that $(v''_{t_0,\varepsilon}\circ\varphi)|F|^{2}e^{-\varphi}$ have a uniform bound on $D_{v}$ independent of $\varepsilon$.
Then the integrals
$\int_{D_v}(v''_{t_0,\varepsilon}\circ\varphi)|F|^{2}e^{-\varphi}d\lambda_{n}$ have a uniform bound independent of $\varepsilon$,
for any given $t_0$ and $D_v$.
As $|(1-v'_{t_0,\varepsilon}\circ\varphi)F|^{2}$ have a uniform bound on $D_{v}$ independent of $\varepsilon$,
it follows that
the integrals $\int_{D_v}|(1-v'_{t_0,\varepsilon}\circ\varphi)F|^{2}d\lambda_{n}$ have a uniform bound independent of $\varepsilon$,
for any given $t_0$ and $D_v$.

As
\begin{equation}
\label{}
\begin{split}
&\int_{ D_v}|F_{v,t_0,\varepsilon}|^{2}d\lambda_{n}
\\&\leq
\int_{ D_v}|F_{v,t_0,\varepsilon}-(1-v'_{t_0,\varepsilon}\circ\varphi)F|^{2}d\lambda_{n}
+\int_{ D_v}|(1-v'_{t_0,\varepsilon}\circ\varphi)F|^{2}d\lambda_{n}
\\&\leq\frac{\mathbf{C}}{e^{A_{t_0}}}\int_{D_v}(v''_{t_0,\varepsilon}\circ\varphi)|F|^{2}e^{-\varphi}d\lambda_{n}
+\int_{ D_v}|(1-v'_{t_0,\varepsilon}\circ\varphi)F|^{2}d\lambda_{n},
\end{split}
\end{equation}
then $\int_{ D_v}|F_{v,t_0,\varepsilon}|^{2}d\lambda_{n}$ have a uniform bound independent of $\varepsilon$.

As $\bar\partial F_{v,t_{0},\varepsilon}=0$ when $\varepsilon\to 0$ and the unit ball of $L^{2}(D_{v})$ is weakly compact,
it follows that the weak limit of some weakly convergent subsequence of
$\{F_{v,t_0,\varepsilon}\}_{\varepsilon}$ gives us a holomorphic $(n,0)$-form $F_{v,t_0}$ on $\Delta^{n}$.
Then we can also choose a subsequence of the weakly convergent subsequence of $\{F_{v,t_0,\varepsilon}\}_{\varepsilon}$,
such that the chosen sequence is uniformly convergent on any compact subset of $D_v$, denoted by $\{F_{v,t_0,\varepsilon}\}_{\varepsilon}$ without ambiguity.

For any given compact subset $K_{0}$ on $D_v$,
$F_{v,t_0,\varepsilon}$, $(1-v'_{t_0,\varepsilon}\circ\varphi)F$ and
$(v''_{t_0,\varepsilon}\circ\varphi)|F|^{2}e^{-\varphi}$ have
uniform bounds on $K_{0}$ independent of $\varepsilon$.

As the integrals $\int_{ K_0}|u_{v,t_0,\varepsilon}|^{2}e^{-\varphi}d\lambda_{n}$
have a uniform bound independent of $\varepsilon\in(0,\frac{1}{8})$,
for any given compact set $K_{0}\subset\subset U_0\cap D_v$ containing $o$,
it follows that
$$(F_{v,t_0}-(1-b_{t_0}(\varphi))F,o)\in \mathcal{I}(\varphi)_{o}.$$

Using the dominated convergence theorem on any compact subset $K$ of
$D_v$ and inequality \ref{equ:3.5}, we obtain
\begin{equation}
\begin{split}
&\int_{K}|F_{v,t_0}-(1-b_{t_0}(\varphi))F|^{2}d\lambda_{n}
\\\leq&\frac{\mathbf{C}}{e^{A_{t_0}}}\int_{D_v}(\mathbb{I}_{\{-t_{0}-1< t<-t_{0}\}}\circ\varphi)|F|^{2}e^{-\varphi}d\lambda_{n}.
\end{split}
\end{equation}

It suffices to find $\eta$ and $\phi$ such that
$(\eta+g^{-1})\leq \mathbf{C}e^{-\varphi_{m}}e^{-\phi}=\mathbf{C}\tilde{\mu}^{-1}$ on $D_v$.
As $\eta=s(-v_{t_0,\varepsilon}\circ\varphi_{m})$ and $\phi=u(-v_{t_0,\varepsilon}\circ\varphi_{m})$,
we have $(\eta+g^{-1}) e^{v_{t_0,\varepsilon}\circ\varphi_{m}}e^{\phi}=(s+\frac{s'^{2}}{u''s-s''})e^{-t}e^{u}\circ(-v_{t_0,\varepsilon}\circ\varphi_{m})$.

Summarizing the above discussion about $s$ and $u$, we are naturally led to a
system of ODEs:
\begin{equation}
\label{GZ}
\begin{split}
&1).\,\,(s+\frac{s'^{2}}{u''s-s''})e^{u-t}=\mathbf{C}, \\
&2).\,\,s'-su'=1,
\end{split}
\end{equation}
where $t\in[0,+\infty)$, and $\mathbf{C}=1$.

It is not hard to solve the ODE system \ref{GZ} (details see the following Remark) and get $u=-\log(1-e^{-t})$ and
$s=\frac{t}{1-e^{-t}}-1$.
It follows that $s\in C^{\infty}((0,+\infty))$ satisfies $s\geq0$, $\lim_{t\to+\infty}u(t)=0$ and
$u\in C^{\infty}((0,+\infty))$ satisfies $u''s-s''>0$.

As $u=-\log(1-e^{-t})$ is decreasing with respect to $t$,
then
$$\frac{\mathbf{C}}{e^{A_{t_0}}}=\frac{1}{\exp\inf_{t\geq t_{0}}u(t)}=\sup_{t\geq t_{0}}\frac{1}{e^{u(t)}}=\sup_{t\geq t_{0}}(1-e^{-t})=1,$$
therefore we are done.
Now we obtain Proposition \ref{p:GZ_JM}.

\begin{Remark}
In fact, we can solve the equation \ref{GZ} as follows:

By $2)$ of equation \ref{GZ}, we have  $su''-s''=-s'u'$. Then $1)$ of equation \ref{GZ} can be change into
$$(s-\frac{s'}{u'})e^{u-t}=\mathbf{C},$$ which is
$$\frac{su'-s'}{u'}e^{u-t}=\mathbf{C}.$$
By $2)$ of equation \ref{GZ}, we have
$$\mathbf{C}=\frac{su'-s'}{u'}e^{u-t}=\frac{-1}{u'}e^{u-t},$$
which is
$$\frac{de^{-u}}{dt}=-u'e^{-u}=\frac{e^{-t}}{\mathbf{C}}.$$ Note that $2)$ of equation \ref{GZ} is equivalent to
$\frac{d(se^{-u})}{dt}=e^{-u}$.
As $s\geq 0$ and $\lim_{t\to+\infty}u=0$,
it follows that the solution
\begin{displaymath}
     \begin{cases}
      u=-\log(1-e^{-t}), \\
      s=\frac{t+e^{-t}-1}{1-e^{-t}},\end{cases}
    \end{displaymath}
\end{Remark}

\subsection{Proof of Proposition \ref{p:GZ_JM_smooth}}
$\\$

For the sake of completeness, we recall our proof in \cite{guan-zhou13p} (see also \cite{guan-zhou13ap}) with some modifications.

As $D_{v}\subset\subset\Delta^{n}\subset \mathbb{C}^{n}$,
then there exist negative smooth plurisubharmonic functions $\{\varphi_{m}\}_{m=1,2,\cdots}$ and
smooth functions $\{\Psi_{m}\}_{m=1,2,\cdots}$
on a neighborhood of $\overline{D}_{v}$,
such that

1). $\{\varphi_{m}+\Psi_{m}\}_{m=1,2,\cdots}$ and $\{\varphi_{m}+2\Psi_{m}\}_{m=1,2,\cdots}$ are negative smooth plurisubharmonic functions;

2). the sequence $\{\varphi_{m}\}_{m=1,2,\cdots}$ is decreasingly convergent to $\varphi$;

3). the sequence $\{\varphi_{m}+\Psi_{m}\}_{m=1,2,\cdots}$ is decreasingly convergent to $\varphi+\Psi$;
$\\$on a smaller neighborhood of $\overline{D}_{v}$,
when $m\to+\infty$.

Let $\eta=s(-v_{t_{0},\varepsilon}\circ\Psi_{m})$ and $\phi=u(-v_{t_{0},\varepsilon}\circ\Psi_{m})$,
where $s\in C^{\infty}((0,+\infty))$ satisfies $s\geq1$, and
$u\in C^{\infty}((0,+\infty))$, such that $u''s-s''>0$, and $s'-u's=1$.

Let $\Phi:=\varphi_{m}+\Psi_{m}+\phi$.

Now let $\alpha=\sum^{n}_{j=1}\alpha_{j}d\bar z^{j}\in Dom_{D_v}
(\bar\partial^*)\cap Ker(\bar\partial)\cap C^\infty_{(0,1)}(\overline {D_v})$.
By Cauchy-Schwarz inequality, it follows that
\begin{equation}
\label{equ:20131130aJM}
\begin{split}
2\mathrm{Re}(\bar\partial^*_\Phi\alpha,\alpha\llcorner(\bar\partial\eta)^\sharp )_{\Omega,\Phi}
\geq
&-\int_{D_v}g^{-1}|\bar{\partial}^{*}_{\Phi}\alpha|^{2}e^{-\Phi}d\lambda_{n}
\\&+
\sum_{j,k=1}^{n}\int_{D_v}
(-g(\partial_{j} \eta)\bar\partial_{k} \eta )\alpha_{\bar{j} }\overline{{\alpha}_{\bar {k}}}e^{-\Phi}d\lambda_{n}.
\end{split}
\end{equation}

Using Lemma \ref{l:lem3} and inequality \ref{equ:20131130aJM},
since $s\geq0$ and $\varphi_{m}$ is a plurisubharmonic function on $\overline{D}_{v}$,
we get

\begin{equation}
\label{equ:4.1JM}
\begin{split}
&\int_{D_v}(\eta+g^{-1})|\bar{\partial}^{*}_{\Phi}\alpha|^{2}e^{-\Phi}d\lambda_{n}
\\\geq&\sum_{j,k=1}^{n}\int_{D_v}
(-\partial_{j}\bar{\partial}_{k}\eta+\eta\partial_{j}\bar{\partial}_{k}\Phi-g(\partial_{j} \eta)\bar\partial_{k} \eta )\alpha_{\bar{j} }\overline{{\alpha}_{\bar{k}}}e^{-\Phi}d\lambda_{n}
\\=&\sum_{j,k=1}^{n}\int_{D_v}
(-\partial_{j}\bar{\partial}_{k}\eta+\eta\partial_{j}\bar{\partial}_{k}\phi+\eta\partial_{j}\bar{\partial}_{k}(\Psi_{m}+\varphi_{m})-g(\partial_{j} \eta)\bar\partial_{k} \eta )\alpha_{\bar{j} }\overline{{\alpha}_{\bar {k}}}e^{-\Phi}d\lambda_{n},
\end{split}
\end{equation}
where $g$ is a positive continuous function on $D_{v}$.
We need some calculations to determine $g$.

We have

\begin{equation}
\label{}
\begin{split}
&\partial_{j}\bar{\partial}_{k}\eta=-s'(-v_{t_0,\varepsilon}\circ \Psi_{m})\partial_{j}\bar{\partial}_{k}(v_{t_0,\varepsilon}\circ \Psi_{m})
\\+&s''(-v_{t_0,\varepsilon}\circ\Psi_{m})\partial_{j}(v_{t_0,\varepsilon}\circ \Psi_{m})
\bar{\partial}_{k}(v_{t_0,\varepsilon}\circ\Psi_{m}),
\end{split}
\end{equation}

and
\begin{equation}
\label{}
\begin{split}
&\partial_{j}\bar{\partial}_{k}\phi=-u'(-v_{t_0,\varepsilon}\circ \Psi_{m})\partial_{j}\bar{\partial}_{k}v_{t_0,\varepsilon}\circ \Psi_{m}
\\+&
u''(-v_{t_0,\varepsilon}\circ \Psi_{m})\partial_{j}(v_{t_0,\varepsilon}\circ \Psi_{m})\bar{\partial}_{k}(v_{t_0,\varepsilon}\circ \Psi_{m})
\end{split}
\end{equation}
for any $j,k$ ($1\leq j,k\leq n$).

We have

\begin{equation}
\label{equ:20131204bJM}
\begin{split}
&\sum_{1\leq j,k\leq n}(-\partial_{j}\bar{\partial}_{k}\eta+\eta\partial_{j}\bar{\partial}_{k}\phi-g(\partial_{j} \eta)
\bar\partial_{k} \eta )\alpha_{\bar{j} }\overline{{\alpha}_{\bar{k}}}
\\=&(s'-su')\sum_{1\leq j,k\leq n}\partial_{j}\bar{\partial}_{k}(v_{t_0,\varepsilon}\circ \Psi_{m})\alpha_{\bar{j} }\overline{{\alpha}_{\bar{k}}}
\\+&((u''s-s'')-gs'^{2})\sum_{1\leq j,k\leq n}\partial_{j}
(-v_{t_0,\varepsilon}\circ \Psi_{m})\bar{\partial}_{k}(-v_{t_0,\varepsilon}\circ\Psi_{m})\alpha_{\bar{j} }\overline{{\alpha}_{\bar{k}}}
\\=&(s'-su')\sum_{1\leq j,k\leq n}((v'_{t_0,\varepsilon}\circ\Psi_{m})\partial_{j}\bar{\partial}_{k}\Psi_{m}+(v''_{t_0,\varepsilon}\circ \Psi_{m})\partial_{j}(\Psi_{m})\bar{\partial}_{k}(\Psi_{m}))\alpha_{\bar{j} }\overline{{\alpha}_{\bar{k}}}
\\+&((u''s-s'')-gs'^{2})\sum_{1\leq j,k\leq n}\partial_{j}
(-v_{t_0,\varepsilon}\circ \Psi_{m})\bar{\partial}_{k}(-v_{t_0,\varepsilon}\circ \Psi_{m})\alpha_{\bar{j} }\overline{{\alpha}_{\bar{k}}}.
\end{split}
\end{equation}
We omit composite item $(-v_{t_0,\varepsilon}\circ \Psi_{m})$ after $s'-su'$ and $(u''s-s'')-gs'^{2}$ in the above equalities.

Since $\varphi_{m}+\Psi_{m}$ and $\varphi_{m}+2\Psi_{m}$ are plurisubharmonic on $\overline{D}_{v}$ and
$0\leq v'_{t_{0},\varepsilon}\circ\Psi_{m}\leq1$,
we have
\begin{equation}
(1-v'_{t_0,\varepsilon}\circ\Psi_{m})\sqrt{-1}\partial\bar\partial(\varphi_{m}+\Psi_{m})+
(v'_{t_0,\varepsilon}\circ\Psi_{m})\sqrt{-1}\partial\bar\partial(\varphi_{m}+2\Psi_{m})\geq 0,
\end{equation}
on $\overline{D}_{v}$, which means that
\begin{equation}
\label{equ:20131204aJM}
\sqrt{-1}\partial\bar\partial(\varphi_{m}+\Psi_{m})+(v'_{t_0,\varepsilon}
\circ\Psi_{m})\sqrt{-1}\partial\bar{\partial}\Psi_{m}\geq 0,
\end{equation}
on $\overline{D}_{v}$.

Let $g=\frac{u''s-s''}{s'^{2}}\circ(-v_{t_0,\varepsilon}\circ \Psi_{m})$.
It follows that $\eta+g^{-1}=(s+\frac{s'^{2}}{u''s-s''})\circ(-v_{t_0,\varepsilon}\circ \Psi_{m})$.

Because of $v'_{t_0,\varepsilon}\geq 0$  and $s'-su'=1$, using inequalities \ref{equ:4.1JM} \ref{equ:20131204aJM} and \ref{equ:20131204bJM},
we have
\begin{equation}
\label{equ:3.1JM}
\begin{split}
\int_{D_v}(\eta+g^{-1})|\bar{\partial}^{*}_{\Phi}\alpha|^{2}e^{-\Phi}d\lambda_{n}
\geq\int_{D_v}(v''_{t_0,\varepsilon}\circ{\Psi_{m}})
\big|\alpha\llcorner(\bar \partial \Psi_{m})^\sharp\big|^2e^{-\Phi}d\lambda_{n}.
\end{split}
\end{equation}

Let $\lambda=\bar{\partial}[(1-v'_{t_0,\varepsilon}(\Psi_{m})){F^{2}}]$.
By the definition of contraction, Cauchy-Schwarz inequality and inequality \ref{equ:3.1JM},
it follows that
\begin{equation}
\begin{split}
&|(\lambda,\alpha)_{D_v,\Phi}|^{2}=|((v''_{t_0,\varepsilon}\circ{\Psi_{m}})\bar\partial\Psi_{m} F^{2},\alpha)_{D_v,\Phi}|^{2}
\\=&|((v''_{t_0,\varepsilon}\circ{\Psi_{m}})F^{2},\alpha\llcorner(\bar\partial\Psi_{m})^\sharp\big)_{D_v,\Phi}|^{2}
\\\leq&\int_{D_v}
(v''_{t_0,\varepsilon}\circ{\Psi_{m}})|F^{2}|^2e^{-\Phi}d\lambda_{n}\int_{D_v}(v''_{t_0,\varepsilon}\circ{\Psi_{m}})
\big|\alpha\llcorner(\bar\partial\Psi_{m})^\sharp\big|^2e^{-\Phi}d\lambda_{n}
\\\leq&
(\int_{D_v}
(v''_{t_0,\varepsilon}\circ{\Psi_{m}})| F^{2}|^2e^{-\Phi}d\lambda_{n})
(\int_{ D_v}(\eta+g^{-1})|\bar{\partial}^{*}_{\Phi}\alpha|^{2}e^{-\Phi}d\lambda_{n}).
\end{split}
\end{equation}

Let $\mu:=(\eta+g^{-1})^{-1}$. Using Lemma \ref{l:lem7},
we have locally $L^{1}$ function $u_{v,t_0,m,\varepsilon}$ on $D_{v}$ such that $\bar{\partial}u_{v,t_0,m,\varepsilon}=\lambda$,
and
\begin{equation}
 \label{equ:3.2JM}
 \begin{split}
 &\int_{ D_v}|u_{v,t_0,m,\varepsilon}|^{2}(\eta+g^{-1})^{-1} e^{-\Phi}d\lambda_{n}
  \leq\int_{D_v}(v''_{t_0,\varepsilon}\circ{\Psi_{m}})|F^{2}|^2e^{-\Phi}d\lambda_{n}.
  \end{split}
\end{equation}

Let $\mu_{1}=e^{v_{t_0,\varepsilon}\circ\Psi_{m}}$, $\tilde{\mu}=\mu_{1}e^{\phi}$.
Assume that we can choose $\eta$ and $\phi$ such that $\tilde{\mu}\leq \mathbf{C}(\eta+g^{-1})^{-1}=\mu$, where $\mathbf{C}=1$.

Note that $v_{t_0,\varepsilon}(\Psi_{m})\geq\Psi_{m}$.
Then it follows that
\begin{equation}
\label{equ:3.8JM}
\begin{split}
\int_{ D_v}|u_{v,t_0,m,\varepsilon}|^{2}e^{-\varphi_{m}} d\lambda_{n}
&\leq\int_{ D_v}|u_{v,t_0,m,\varepsilon}|^{2}\mu_{1}e^{\phi} e^{-\Psi_{m}-\varphi_{m}-\phi}d\lambda_{n}
\\&=\int_{ D_v}|u_{v,t_0,m,\varepsilon}|^{2}\tilde{\mu}e^{-\Phi}d\lambda_{n}.
\end{split}
\end{equation}

Using inequalities \ref{equ:3.2JM} and \ref{equ:3.8JM},
we obtain that
$$\int_{ D_v}|u_{v,t_0,m,\varepsilon}|^{2}e^{-\varphi_{m}} d\lambda_{n}\leq\mathbf{C}\int_{D_v}
(v''_{t_0,\varepsilon}\circ\Psi_{m})|F^{2}|^2e^{-\Phi}d\lambda_{n},$$
under the assumption $\tilde{\mu}\leq\mathbf{C} (\eta+g^{-1})^{-1}$.

As $-v_{t_0,\varepsilon}\circ{\Psi_{m}}(\overline{D_{v}})\subset\subset(-\infty,t_{0}+1)$,
then it is clear that
\begin{equation}
\label{equ:20131130bJM}
-v_{t_0,\varepsilon}\circ{\Psi_{m}}(\overline{D_{v}})\subset(-\infty,K_{t_{0}})
\end{equation}
where $K_{t_{0}}$ is independent of $m$ and $\varepsilon\in (0,\frac{1}{8}B_{0})$.
As $u$ is positive and smooth on $(-\infty,+\infty)$,
it follows that $\phi$ is uniformly bounded on $\overline{D}_{v}$ independent of $m$.

As $Supp(v''_{t_{0},\varepsilon})\subset\subset(-t_{0}-B_{0},-t_{0})$
and $|F^{2}|^2e^{-\varphi_{m}}\leq|F^{2}|^2e^{-\varphi}\leq 1$,
then it is clear that $(v''_{t_0,\varepsilon}\circ{\Psi_{m}})|F^{2}|^2e^{-\varphi_{m}-\Psi_{m}}$ are uniformly bounded on $\overline{D}_{v}$ independent of $m$.
Therefore the integrals
$\int_{D_v}(v''_{t_0,\varepsilon}\circ{\Psi_{m}})|F^{2}|^2e^{-\Phi}d\lambda_{n}$ are uniformly bounded independent of $m$,
for any given $v$, $t_0$, $\varepsilon$.

By weakly compactness of the unit ball of $L^{2}_{\varphi}(D_{v})$ and dominated convergence theorem, when $m\to+\infty$,
it follows that the weak limit of some weakly convergent subsequence of
$\{u_{v,t_{0},m,\varepsilon}\}_{m}$ gives function $u_{v,t_0,\varepsilon}$ on $D_v$ satisfying
\begin{equation}
\label{equ:3.3JM}
\int_{ D_v}|u_{v,t_0,\varepsilon}|^{2}e^{-\varphi}d\lambda_{n}\leq\frac{\mathbf{C}}{e^{A_{t_0}}}\int_{D_v}
(v''_{t_0,\varepsilon}\circ{\Psi})|F^{2}|^2e^{-\varphi-\Psi}d\lambda_{n},
\end{equation}
where $A_{t_0}:=\inf_{t\geq t_0}\{u(t)\}$.

Let $F_{v,t_0,m,\varepsilon}:=(1-v'_{t_0,\varepsilon}(\Psi_{m}))F^{2}-u_{v,t_0,m,\varepsilon}$,
which is a holomorphic function on $D_{v}$.

As $|F^{2}|^{2}e^{-\varphi_{m}}\leq|F^{2}|^{2}e^{-\varphi} \leq 1$,
then the integrals
$\int_{ D_v}|(1-v'_{t_0,\varepsilon}(\Psi_{m}))F^{2}|^{2}e^{-\varphi_{m}} d\lambda_{n}$ have a uniform bound independent of $m$.
Recall that the integrals
$$\int_{ D_v}|u_{v,t_{0},m,\varepsilon}|^{2}e^{-\varphi_{m}} d\lambda_{n}$$
have a uniform bound independent of $m$,
then the integrals
$$\int_{ D_v}|F_{v,t_0,m,\varepsilon}|^{2}e^{-\varphi_{m}} d\lambda_{n}$$
have a uniform bound independent of $m$.
Therefore we can choose a subsequence of $\{F_{v,t_0,m,\varepsilon}\}_{m=1,2,\cdots}$ from the chosen weakly convergent subsequence of $$(1-v'_{t_0,\varepsilon}(\Psi_{m}))F^{2}-u_{v,t_0,m,\varepsilon},$$
such that the subsequence is uniformly convergent on any compact subset of $D_{v}$,
and we still denote the subsequence by $\{F_{v,t_0,m,\varepsilon}\}_{m=1,2,\cdots}$ without ambiguity.
Let $$F_{v,t_0,\varepsilon}:=\lim_{m\to\infty}F_{v,t_0,m,\varepsilon}.$$

By above arguments,
it follows that the right hand side of inequality \ref{equ:3.2JM} are uniformly bounded independent of $m$ and $\varepsilon\in(0,\frac{1}{8}B_{0})$.
By inequality \ref{equ:3.2JM},
it follows that the integrals
$$\int_{ D_v}|u_{v,t_0,m,\varepsilon}|^{2}(\eta+g^{-1})^{-1} e^{-\phi-\varphi_{m}-\Psi_{m}}d\lambda_{n}$$
are uniformly bounded independent of $m$ and $\varepsilon\in(0,\frac{1}{8}B_{0})$.

Using inequality \ref{equ:20131130bJM},
we obtain that
$$(\eta+g^{-1})^{-1}=(s(-v_{t_{0},\varepsilon}\circ\Psi_{m})+\frac{s'^{2}}{u''s-s''}\circ(-v_{t_0,\varepsilon}\circ\Psi_{m}))^{-1}$$
and
$e^{-\phi}=e^{-u(-v_{t_{0},\varepsilon}\circ\Psi_{m})}$ have positive uniform lower bounds independent of $m$ and $\varepsilon\in(0,\frac{1}{8}B_{0})$.
Then the integrals
$$\int_{ K_0}|u_{v,t_0,m,\varepsilon}|^{2}e^{-\varphi_{m}-\Psi_{m}}d\lambda_{n}$$
have a uniform upper bound independent of $m$ and $\varepsilon\in(0,\frac{1}{8}B_{0})$
for any given compact set $K_{0}\subset\subset D_v$.

As $$Supp(v'_{t_0,\varepsilon}(\Psi_{m}))\subset \{\Psi_{m}>-t_{0}-1\},$$
it follows that
$$|v'_{t_0,\varepsilon}(\Psi_{m})|^{2}e^{-\Psi_{m}}\leq e^{t_{0}+1}.$$

Furthermore, as
$$|F|^{4}e^{-\varphi_{m}}\leq|F|^{4}e^{-\varphi}=e^{-2\max\{\psi-\log|F|^{2},0\}}\leq1$$
then the integrals
$$\int_{ K_0}|v'_{t_0,\varepsilon}(\Psi_{m})F^{2}|^{2}e^{-\varphi_{m}-\Psi_{m}}d\lambda_{n}$$
have a uniform upper bound independent of $m$ and $\varepsilon\in(0,\frac{1}{8}B_{0})$
for any given compact set $K_{0}\subset\subset D_v$.
Therefore the integrals
\begin{equation}
\label{equ:20131211JM}\int_{ K_0}|F_{v,t_0,m,\varepsilon}-F^{2}|^{2}e^{-\varphi_{m}-\Psi_{m}}d\lambda_{n}
\end{equation}
have a uniform upper bound independent of $m$ and $\varepsilon\in(0,\frac{1}{8}B_{0})$
for any given compact set $K_{0}\subset\subset D_v$.

As
$\varphi_{m'}+\Psi_{m'}\leq\varphi_{m}+\Psi_{m}$ where $m'\geq m$,
it follows that
$$|F_{v,t_0,m',\varepsilon}-F^{2}|^{2}e^{-(\varphi_{m}+\Psi_{m})}\leq|F_{v,t_0,m',\varepsilon}-F^{2}|^{2}e^{-(\varphi_{m'}+\Psi_{m'})}.$$
By inequality \ref{equ:20131211JM}, it follows that for any given compact set $K_{0}\subset\subset  D_v$,
the integrals $\int_{ K_0}|F_{v,t_0,m',\varepsilon}-F^{2}|^{2}e^{-(\varphi_{m}+\Psi_{m})}d\lambda_{n}$
have a uniform bound independent of $m$ and $m'$.
Therefore
for any given compact set $K_{0}\subset\subset\cap D_v$,
the integrals
$\int_{ K_0}|F_{v,t_0,\varepsilon}-F^{2}|^{2}e^{-(\varphi_{m}+\Psi_{m})}d\lambda_{n}$ have a uniform bound independent of $m$ and $\varepsilon\in(0,\frac{1}{8}B_{0})$.
It is clear that the integrals
\begin{equation}
\label{equ:20131211aJM}\int_{ K_0}|F_{v,t_0,\varepsilon}-F^{2}|^{2}e^{-\varphi-\Psi}d\lambda_{n}
\end{equation}
have a uniform upper bound independent of $\varepsilon\in(0,\frac{1}{8}B_{0})$
for any given compact set $K_{0}\subset\subset D_v$.

In summary,
we have $|F_{v,t_0,\varepsilon}-F^{2}|^{2}e^{-\varphi-\Psi}$ is integrable near $o$.
That is to say
$$(F_{v,t_0,\varepsilon}-F^{2},o)\in\mathcal{I}(\varphi+\Psi)_{o}.$$
By inequality \ref{equ:3.3JM},
it follows that
$F_{v,t_0,\varepsilon}$ is a holomorphic function on $D_{v}$ satisfying $(F_{v,t_0,\varepsilon}-F^{2},o)\in\mathcal{I}(\varphi)_{o}$ and
\begin{equation}
\label{equ:3.5JM}
\begin{split}
&\int_{ D_v}|F_{v,t_0,\varepsilon}-(1-v'_{t_0,\varepsilon}\circ\Psi)F^{2}|^{2}e^{-\varphi}d\lambda_{n}
\\\leq&\frac{\mathbf{C}}{e^{A_{t_0}}}\int_{D_v}(v''_{t_0,\varepsilon}\circ\Psi)|F^{2}|^{2}e^{-\varphi-\Psi}d\lambda_{n}.
\end{split}
\end{equation}

It suffices to find $\eta$ and $\phi$ such that
$(\eta+g^{-1})\leq \mathbf{C}e^{-\Psi_{m}}e^{-\phi}=\mathbf{C}\tilde{\mu}^{-1}$ on $D_v$.
As $\eta=s(-v_{t_0,\varepsilon}\circ\Psi_{m})$ and $\phi=u(-v_{t_0,\varepsilon}\circ\Psi_{m})$,
we have $(\eta+g^{-1}) e^{v_{t_0,\varepsilon}\circ\Psi_{m}}e^{\phi}=(s+\frac{s'^{2}}{u''s-s''})e^{-t}e^{u}\circ(-v_{t_0,\varepsilon}\circ\Psi_{m})$.

Summarizing the above discussion about $s$ and $u$, we are naturally led to a
system of ODEs:
\begin{equation}
\label{GZJM}
\begin{split}
&1).\,\,(s+\frac{s'^{2}}{u''s-s''})e^{u-t}=\mathbf{C}, \\
&2).\,\,s'-su'=1,
\end{split}
\end{equation}
where $t\in[0,+\infty)$, and $\mathbf{C}=1$.

It is not hard to solve the ODE system \ref{GZJM} (details see the following Remark)
and get $u=-\log(2-e^{-t})$ and
$s(t)=\frac{2t+e^{-t}}{2-e^{-t}}$.
It follows that $s\in C^{\infty}((0,+\infty))$ satisfies $s\geq1$, $u'\leq 0$ and
$u\in C^{\infty}((0,+\infty))$ satisfies $u''s-s''>0$.

As $u=-\log(2-e^{-t})$ is decreasing with respect to $t$,
then
$$\frac{\mathbf{C}}{e^{A_{t_0}}}=\frac{1}{\exp\inf_{t\geq t_{0}}u(t)}=\sup_{t\geq t_{0}}\frac{1}{e^{u(t)}}=\sup_{t\geq t_{0}}(2-e^{-t})=2,$$
for any $t_{0}\geq0$,
therefore we are done.
Now we obtain Proposition \ref{p:GZ_JM_smooth}.

\begin{Remark}
In fact, we can solve the equation \ref{GZJM} as follows:

By $2)$ of equation \ref{GZJM}, we have  $su''-s''=-s'u'$. Then $1)$ of equation \ref{GZJM} can be change into
$$(s-\frac{s'}{u'})e^{u-t}=\mathbf{C},$$ which is
$$\frac{su'-s'}{u'}e^{u-t}=\mathbf{C}.$$
By $2)$ of equation \ref{GZJM}, we have
$$\mathbf{C}=\frac{su'-s'}{u'}e^{u-t}=\frac{-1}{u'}e^{u-t},$$
which is
$$\frac{de^{-u}}{dt}=-u'e^{-u}=\frac{e^{-t}}{\mathbf{C}}.$$ Note that $2)$ of equation \ref{GZJM} is equivalent to
$\frac{d(se^{-u})}{dt}=e^{-u}$.
As $s\geq 1$ and $\lim_{t\to+\infty}u=0$,
it follows that the solution
\begin{displaymath}
     \begin{cases}
      u=-\log(2-e^{-t}), \\
      s=\frac{2t+e^{-t}}{2-e^{-t}},\end{cases}
    \end{displaymath}
\end{Remark}

\subsection{Proof of Proposition \ref{p:GZ_JM201312}}
$\\$

Given $t_0$ and $D_{v}$,
it is clear that $(v''_{t_0,\varepsilon}\circ\Psi)|F^{2}|^{2}e^{-\varphi-\Psi}$ have a uniform bound on $D_{v}$ independent of $\varepsilon$.
Then the integrals
$\int_{D_v}(v''_{t_0,\varepsilon}\circ\Psi)|F^{2}|^{2}e^{-\varphi-\Psi}d\lambda_{n}$ have a uniform bound independent of $\varepsilon$,
for any given $t_0$ and $D_v$.

As $|(v'_{t_0,\varepsilon}\circ\Psi)F^{2}|^{2}e^{-\varphi}$ have a uniform bound on $D_{v}$ independent of $\varepsilon$,
it follows that the integrals $\int_{D_v}|(1-v'_{t_0,\varepsilon}\circ\Psi)F^{2}|^{2}e^{-\varphi}d\lambda_{n}$ have a uniform bound independent of $\varepsilon$,
for any given $t_0$ and $D_v$.

As
\begin{equation}
\label{}
\begin{split}
&\int_{ D_v}|F_{v,t_0,\varepsilon}|^{2}e^{-\varphi}d\lambda_{n}
\\&\leq
\int_{ D_v}|F_{v,t_0,\varepsilon}-(1-v'_{t_0,\varepsilon}\circ\Psi)F^{2}|^{2}e^{-\varphi}d\lambda_{n}
+\int_{ D_v}|(1-v'_{t_0,\varepsilon}\circ\Psi)F^{2}|^{2}e^{-\varphi}d\lambda_{n}
\\&\leq\frac{\mathbf{C}}{e^{A_{t_0}}}\int_{D_v}(v''_{t_0,\varepsilon}\circ\Psi)|F^{2}|^{2}e^{-\varphi-\Psi}d\lambda_{n}
+\int_{ D_v}|(1-v'_{t_0,\varepsilon}\circ\Psi)F^{2}|^{2}e^{-\varphi}d\lambda_{n},
\end{split}
\end{equation}
then the integrals $\int_{ D_v}|F_{v,t_0,\varepsilon}|^{2}e^{-\varphi}d\lambda_{n}$ have a uniform bound independent of $\varepsilon$.

As $\bar\partial F_{v,t_{0},\varepsilon}=0$ when $\varepsilon\to 0$ and the unit ball of $L^{2}_{\varphi}(D_{v})$ is weakly compact,
it follows that the weak limit of some weakly convergent subsequence of
$\{F_{v,t_0,\varepsilon}\}_{\varepsilon}$ gives us a holomorphic function $F_{v,t_0}$ on $\Delta^{n}$.
Then we can also choose a subsequence of the weakly convergent subsequence of $\{F_{v,t_0,\varepsilon}\}_{\varepsilon}$,
such that the chosen sequence is uniformly convergent on any compact subset of $D_v$, denoted by $\{F_{v,t_0,\varepsilon}\}_{\varepsilon}$ without ambiguity.
For any given compact subset $K_{0}$ on $D_v$, $F_{v,t_0,\varepsilon}$,
$|(1-v'_{t_0,\varepsilon}\circ\Psi)F^{2}|^{2}e^{-\varphi}$ and $(v''_{t_0,\varepsilon}\circ\Psi)|F^{2}|^{2}e^{-\varphi-\Psi}$
have uniform bounds on $K_{0}$ independent of $\varepsilon$.

By inequality \ref{equ:20131211aJM},
it follows that
$$(F_{v,t_0}-F^{2},o)\in \mathcal{I}(\varphi+\Psi)_{o}.$$
Using the dominated convergence theorem on any compact subset $K$ of $D_v$ and Proposition \ref{p:GZ_JM_smooth},
we obtain
\begin{equation}
\begin{split}
&\int_{K}|F_{v,t_0}-(1-b_{t_0}(\Psi))F^{2}|^{2}e^{-\varphi}d\lambda_{n}
\\\leq&2\int_{D_v}(\mathbb{I}_{\{-t_{0}-1< t<-t_{0}\}}\circ\Psi)|F^{2}|^{2}e^{-\varphi-\Psi}d\lambda_{n}.
\end{split}
\end{equation}
Then Proposition \ref{p:GZ_JM201312} has thus been proved.

\section{Proofs of the main results}

\subsection{Proof of Theorem \ref{t:strong1015}}
$\\$

We will prove Theorem \ref{t:strong1015}
by the methods of induction and contradiction, and
by using dynamically $L^{2}$ extension theorem with negligible weight.

Let $\{p_{j}\}_{j=1,2,\cdots}$ be a sequence of positive numbers which are strictly decreasingly convergent to $1$, when $j$ goes to infinity.

\subsubsection{Step 1: Theorem \ref{t:strong1015} for dimension 1 case}
$\\$

We first consider Theorem \ref{t:strong1015} for dimension 1 case,
which is elementary but
revealing.

We choose $r_{0}$ small enough,
such that $\{F=0\}\cap\Delta_{r_0}\subset \{o\}$.

As $\int_{\Delta}|F|^{2}e^{-\varphi}d\lambda_{1}<+\infty$,
by Lemma \ref{l:open_a},
we have
$$\liminf_{A\to+\infty}\mu(\{|F|^{2}e^{-\varphi}> A\})u(A)=0.$$

It is clear that,
for any give $B>0$,
there exists $A>B$ and $z_{A}\in\Delta_{u(A)^{-1/2}}$,
such that $e^{-\varphi(z_{A})}|F(z_{A})|^{2}\leq A$.
We can assume that $A>10$.

Let $\psi=-\log2$,
then $\log|z'-z_{A}|+\psi<0$.

Using Theorem \ref{t:guan-zhou12} on $\Delta$,
we obtain holomorphic function $F_{A}$ on $\Delta$ for each $A$ and $p_{j_A}\varphi$ ($j_{A}\in\{1,2,\cdots\}$),
such that
$F_{A}|_{z_{A}}=F(z_{A})$,
and
\begin{equation}
\label{equ:0927a}
\int_{\Delta}|F_{A}|^{2}e^{-p_{j_{A}}\varphi}d\lambda_{1}<8\pi A.
\end{equation}

By the negativeness of $\varphi$ and $p_{j_{A}}\varphi$,
it follows that
\begin{equation}
\label{equ:0927b}\int_{\Delta}|F_{A}|^{2}d\lambda_{1}<8\pi A.
\end{equation}

Assume Theorem \ref{t:strong1015} for $n=1$ is not true.
Therefore
$$\int_{\Delta_{r}}|F|^{2}e^{-p_{j}\varphi}d\lambda=+\infty,$$
for any $r>0$ and any $j\in\{1,2,\cdots\}$.

Since $\{F=0\}\cap\Delta_{r_0}\subset \{o\}$,
then it follows from inequality \ref{equ:0927a}
that one can derive that $F$/$F_A$ is unbounded.
Otherwise,
the boundedness would imply the finiteness of the integral of $|F|^2e^{-p_{j}\varphi}$,
according to inequality \ref{equ:0927a}.
This contradicts to the assumption.
Then there exists a holomorphic function $h_{A}$ on $\Delta_{r_0}$,
such that

1). $F_{A}|_{\Delta_{r_0}}=F|_{\Delta_{r_0}}h_{A}$;

2). $h_{A}(o)=0$;

3). $h_{A}(z_{A})=1$.

By Lemma \ref{l:open_b},
it follows that
$$\int_{\Delta_{r_0}}|F_{A}|^{2}d\lambda_{1}>C_{1}u(A),$$
where $C_{1}$ is independent of $A$.

It contradicts to
$$\int_{\Delta}|F_{A}|^{2}d\lambda_{1}<8\pi A.$$

We have thus proved Theorem \ref{t:strong1015} for $n=1$.

\subsubsection{Step 2:  Theorem \ref{t:strong1015} for $n=k$}
$\\$

Assume Theorem \ref{t:strong1015}
$$\int_{\Delta^{k}_{r}}|F|^{2}e^{-\varphi}d\lambda_{k}<+\infty,$$
for some $r>0$,
and
$$\int_{\Delta^{k}_{r}}|F|^{2}e^{-p_{j}\varphi}d\lambda_{k}=+\infty,$$
for any $r>0$ and $j\in\{1,2,\cdots\}$.

Then the germ of the holomorphic function $F$ is in $\mathcal{I}(\varphi)_{o}$
but is not in $(\cup_{j=1}^{\infty}\mathcal{I}(p_{j}\varphi))_{o}$.

Using Proposition \ref{r:curve},
we have a germ of analytic curve $\gamma$ through $o$
satisfying $\{F|_{\gamma}=0\}=\{o\}$,
such that
for any germ of holomorphic function $g$ in $\cup_{j=1}^{\infty}\mathcal{I}(p_{j}\varphi)$,
and we also have a holomorphic function $h_{g}$ on $\gamma$ satisfying
$$h_{g}|_{o}=0,$$
such that
\begin{equation}
\label{equ:infact}
g|_{\gamma}=F|_{\gamma}h_{g}.
\end{equation}

Then we can choose biholomorphic map $\imath$ from a neighborhood of
$\overline{\Delta'\times\Delta''}$ to a neighborhood $V_{o}\subset \Delta^{k}$ of $o$,
which is small enough, with origin keeping $\imath(o)=o$,
such that

1), $\imath^{-1}(\gamma)$ is a closed analytic curve of the neighborhood of $\overline{\Delta'\times\Delta''}$;

2), $\imath^{-1}(\gamma)$ satisfies the parametrization property as analytic curve $\mathcal{C}$ in Remark \ref{r:para}.

Note that
$$\int_{V_{o}}|F|^{2}e^{-p_{j}\varphi}d\lambda_{n}=
\int_{\Delta'\times\Delta''}|\imath^{*}(F)|^{2}e^{-\imath^{*}(p_{j}\varphi)}\imath^{*}(d\lambda_{n}).$$
Then
$$\int_{\Delta^{k}_{r}}|F|^{2}e^{-p_{j}\varphi}d\lambda_{n}=+\infty,$$
for any $r>0$ and any $j\in\{1,2,\cdots\}$, is equivalent to
$$\int_{\Delta'_{r}\times\Delta''_{r}}|\imath^{*}(F)|^{2}e^{-\imath^{*}(p_{j}\varphi)}d\lambda_{n}=+\infty,$$
for any $r>0$ and any $j\in\{1,2,\cdots\}$.

As $\int_{\Delta'\times\Delta''}|\imath^{*}(F)|^{2}e^{-\imath^{*}(\varphi)}d\lambda_{k}<+\infty$,
it follows from Lemma \ref{l:open_a} that
$$\liminf_{A\to+\infty}\mu(\{z_{1}|\int_{\pi^{-1}(z_{1})}|\imath^{*}(F)|^{2}e^{-\imath^{*}(\varphi)}d\lambda_{k-1}> A\})u(A)=0,$$
where $\pi$ is the projection in Remark \ref{r:para}.

It is clear that
for any give $B>0$,
there exists $A>B$,
such that
$$\{z_{1}|\int_{\pi^{-1}(z_{1})}|\imath^{*}(F)|^{2}e^{-\imath^{*}(\varphi)}d\lambda_{k-1}> A\}$$
cannot contain
$\Delta_{u(A)^{-1/2}}$.

As $\lambda_{n}(\{\imath^{*}(p_{1}\varphi)=-\infty\})=0$,
then $\lambda_{1}(\{z'|\imath^{*}(p_{1}\varphi)|_{\pi^{-1}(z')}\equiv -\infty\})=0$.
Then
for any given $B>0$,
there exists $A>B$ and $z_{A}\in\Delta_{u(A)^{-1/2}}$,
satisfying
$$\int_{\pi^{-1}(z_{A})}|\imath^{*}(F)|^{2}e^{-\imath^{*}(\varphi)}d\lambda_{k-1}\leq A,$$
such that
$$\imath^{*}(p_{1}\varphi)|_{\pi^{-1}(z_{A})}\not\equiv -\infty.$$
We assume that $A>e^{10}$.

\subsubsection{Using dynamically $L^2$ extension theorems with negligible weight.}
$\\$

As Theorem \ref{t:strong1015} for $n=k-1$ holds,
and $\imath^{*}(p_{1}\varphi)|_{\pi^{-1}(z_{A})}\not\equiv -\infty$,
then there exists $j_{A}\in\{1,2,\cdots\}$,
such that $$\int_{\pi^{-1}(z_{A})}|\imath^{*}(F)|^{2}e^{-\imath^{*}(p_{j_{A}}\varphi)}d\lambda_{k-1}<2A.$$

Let $\psi=-\log2$,
then $\log|z'-z_{A}|+\psi<0$.
By Theorem \ref{t:guan-zhou12} on $\Delta'\times\Delta''$,
we obtain a holomorphic function $F_{A}$ on $\Delta'\times\Delta''$ for each $A$,
such that
$F_{A}|_{\pi^{-1}(z_{A})}=\imath^{*}(F)|_{\pi^{-1}(z_{A})}$,
and
$$\int_{\Delta'\times\Delta''}|F_{A}|^{2}e^{-\imath^{*}(p_{j_{A}}\varphi)}d\lambda_{k}<8\pi A.$$

It follows form equality \ref{equ:infact},
that there exists a holomorphic function on $\gamma$,
denoted by $h_{A}$,
such that
\begin{equation}
\label{equ:infact1}
\imath_{*}F_{A}|_{\gamma}=F|_{\gamma}h_{A},
\end{equation}
therefore,
\begin{equation}
\label{equ:infact1}
F_{A}|_{\imath^{-1}(\gamma)}=\imath^{*}(F)|_{\imath^{-1}(\gamma)}\imath^{*}(h_{A}),
\end{equation}
where $h_{A}(0)=0$, $\imath^{*}(h_{A})(\imath^{-1}(\gamma)\cap \pi^{-1}(z_{A}))=1$.

By the negativeness of $\varphi$,
it is clear that
$$\int_{\Delta'\times\Delta''}|F_{A}|^{2}d\lambda_{1}<8\pi A.$$

Using equality \ref{equ:infact1},
the condition $|z_{A}|<u(A)^{-\frac{1}{2}}$,
and Lemma \ref{l:open_sing},
we have
$$\int_{\imath^{-1}(\gamma)}|F_{A}|^{2}\pi|_{\mathcal{C}_{S}}^{*}d\lambda_{\Delta'}>C_{1}u(A),$$
where $C_{1}>0$ is independent of $A$ and $F_{A}$.
In our use of Lemma \ref{l:open_sing}, $z_{A}$ corresponds to $a$ in Lemma \ref{l:open_sing},
the function $\imath^{*}(h_{A})$ corresponds to $f_{a}$ in Lemma \ref{l:open_sing},
which does not correspond to $h$ in Lemma \ref{l:open_sing}
(actually $\imath^{*}(F)|_{\imath^{-1}(\gamma)}$ corresponds to to $h$ in Lemma \ref{l:open_sing}).

Using Lemma \ref{l:approx.L2},
we obtain
$$\int_{\Delta'\times\Delta''}|F_{A}|^{2}d\lambda_{k}\geq C_{3}\int_{\imath^{-1}(\gamma)}|F_{A}|^{2}
\pi|_{\mathcal{C}_{S}}^{*}d\lambda_{\Delta'},$$
where $C_{3}>0$ is independent of $A$ and $F_{A}$.

Therefore
$$\int_{\Delta'\times\Delta''}|F_{A}|^{2}d\lambda_{k}\geq C_{1} C_{3}u(A),$$
which contradicts to
$$\int_{\Delta'\times\Delta''}|F_{A}|^{2}d\lambda_{k}<8\pi A,$$
for $A$ large enough.

We have thus proved Theorem \ref{t:strong1015} for
$n=k$.

The proof of Theorem \ref{t:strong1015} is thus complete.

\subsubsection{Some remarks of Theorem \ref{t:strong1015}}
$\\$

Let $\varphi$ be a negative
plurisubharmonic function on
$\Delta^{n}\subset\mathbb{C}^{n}$,
and $\{\psi_{j}\}_{j=1,2,\cdots}$ be a sequence of plurisubharmonic
functions on $\Delta^{n}$, which is increasingly convergent to
$\varphi$ on $\Delta^{n}$, when $j\to\infty$.

Without loss of generality, one can assume that $\psi_{1}\not\equiv-\infty$.

In the proof of Theorem \ref{t:strong1015}, replacing $p_{j}\varphi$ by $\psi_{j}$, and $p_{j_{A}}\varphi$ by $\psi_{j_{A}}$,
one can obtain:

Let $F$ be a holomorphic function on $\Delta^{n}$,
such that
$$\int_{\Delta^{n}}|F|^{2}e^{-\varphi}d\lambda_{n}<+\infty.$$
Then there exists a
number $j_{0}\geq 1$, such that
$$\int_{\Delta^{n}_{r}}|F|^{2}e^{-\psi_{j_{0}}}d\lambda_{n}<+\infty,$$
for some $r\in(0,1)$.

That is to say:
\begin{equation}
\label{equ:2014113a}
\cup_{j=1}^{\infty}\mathcal{I}(\psi_{j})=\mathcal{I}(\varphi).
\end{equation}

In particular, let $\psi_{j}=\varphi+\frac{1}{j}\varphi_{0}$ in equality \ref{equ:2014113a},
then we get the following modified version of the strong openness conjecture which was conjectured in  \cite{kim10}:

Let $\varphi$ be a negative plurisubharmonic function on $\Delta^{n}\subset\mathbb{C}^{n}$,
and $\varphi_{0}\not\equiv-\infty$ be a negative plurisubharmonic function on $\Delta^{n}$.
Then
$$\cup_{\varepsilon>0}\mathcal{I}(\varphi+\varepsilon\varphi_{0})=\mathcal{I}(\varphi).$$


\subsection{Proof of Theorem \ref{t:GZ_JM}}
$\\$

We prove Theorem \ref{t:GZ_JM} by contradiction:
if Theorem \ref{t:GZ_JM} is not true,
then there exists $t_{j}\to +\infty$ $(j\to+\infty)$,
such that
\begin{equation}
\label{equ:20131013a}
\lim_{j\to\infty}\int_{\Delta^{n}}\mathbb{I}_{\{-t_{j}-1<\varphi<-t_{j}\}}|F|^{2}e^{-\varphi}d\lambda_{n}=0.
\end{equation}

Let $D_{v}$ be a strongly pseudoconvex domain relatively compact in $\Delta^{n}$ containing $o$,
it follows that
\begin{equation}
\label{equ:20131013}
\begin{split}
\int_{\Delta^{n}}\mathbb{I}_{\{-t_{j}-1<\varphi<-t_{j}\}}|F|^{2}e^{-\varphi}d\lambda_{n}
\geq \int_{D_v}\mathbb{I}_{\{-t_{j}-1<\varphi<-t_{j}\}}|F|^{2}e^{-\varphi}d\lambda_{n}.
\end{split}
\end{equation}

According to equality \ref{equ:20131013a} and inequality \ref{equ:20131013},
it follows that
\begin{equation}
\label{equ:20131120a}
\begin{split}\lim_{j\to+\infty}\int_{D_v}\mathbb{I}_{\{-t_{j}-1<\varphi<-t_{j}\}}|F|^{2}e^{-\varphi}d\lambda_{n}=0.
\end{split}
\end{equation}

By Proposition \ref{p:GZ_JM},
it follows that there exists $F_{v,t_{j}}$,
which is a holomorphic function on $D_{v}$
satisfying:

\begin{equation}
\label{equ:13.10.11a}
\begin{split}
&\int_{ D_v}|F_{v,t_j}-(1-b_{t_j}(\varphi))F|^{2}d\lambda_{n}
\\\leq&\int_{D_v}\mathbb{I}_{\{-t_{j}-1<\varphi<-t_{j}\}}|F|^{2}e^{-\varphi}d\lambda_{n}.
\end{split}
\end{equation}
and
$$(F_{v,t_{j}}-F,o)\in\mathcal{I}(\varphi)_{o}.$$

As $|(1-b_{t_j}(\varphi))F|$ on $D_{v}$ have a uniform bound independent of $j$,
then $\int_{ D_v}|(1-b_{t_j}(\varphi))F|^{2}d\lambda_{n}$ have a uniform bound independent of $j$.

According to equality \ref{equ:20131120a} and inequality \ref{equ:13.10.11a},
it follows that $\int_{ D_v}|F_{v,t_j}-(1-b_{t_j}(\varphi))F|^{2}d\lambda_{n}$ have a uniform bound independent of $j$.

Using
\begin{equation}
\label{equ:20131120b}
\begin{split}
(\int_{ D_v}|F_{v,t_j}|^{2}d\lambda_{n})^{\frac{1}{2}}&\leq
(\int_{ D_v}|F_{v,t_j}-(1-b_{t_j}(\varphi))F|^{2}d\lambda_{n})^{\frac{1}{2}}\\&+(\int_{ D_v}|(1-b_{t_j}(\varphi))F|^{2}d\lambda_{n})^{\frac{1}{2}},
\end{split}
\end{equation}
we have $\int_{ D_v}|F_{v,t_j}|^{2}d\lambda_{n}$ have a uniform bound independent of $j$.
Then there is a subsequence of $F_{v,t_j}$ denoted by $F_{v,t_j}$ without ambiguity,
which is convergent to a holomorphic function $F_{v}$ uniformly on any compact subset of $D_{v}$.
Then for any $K\subset\subset D_{v}$,
we have
\begin{equation}
\label{equ:20131120e}
\begin{split}
\int_{ K}|F_{v}|^{2}d\lambda_{n}\leq\liminf_{j\to+\infty}\int_{ D_v}|F_{v,t_j}|^{2}d\lambda_{n}.
\end{split}
\end{equation}
Therefore
\begin{equation}
\label{equ:20131120f}
\begin{split}
\int_{D_v}|F_{v}|^{2}d\lambda_{n}\leq\liminf_{j\to+\infty}\int_{ D_v}|F_{v,t_j}|^{2}d\lambda_{n}.
\end{split}
\end{equation}

As $\{(1-b_{t_j}(\varphi))F\}_{j=1,2,\cdots}$ goes to zero when $j\to+\infty$ and
$|(1-b_{t_j}(\varphi))F|$ on $D_{v}$ have a uniform bound independent of $j$,
by the Lebesgue dominated convergence theorem£¬
it follows that
\begin{equation}
\label{equ:20131120c}
\begin{split}
\lim_{j\to+\infty}\int_{ D_v}|(1-b_{t_j}(\varphi))F|^{2}d\lambda_{n}=0.
\end{split}
\end{equation}

According to equality \ref{equ:20131120a}, inequality \ref{equ:13.10.11a},
it follows that
\begin{equation}
\label{equ:20131120d}
\begin{split}
\lim_{j\to+\infty}\int_{ D_v}|F_{v,t_j}-(1-b_{t_j}(\varphi))F|^{2}d\lambda_{n}=0.
\end{split}
\end{equation}

Using equality \ref{equ:20131120c}, equality \ref{equ:20131120d} and inequality \ref{equ:20131120b},
we have
\begin{equation}
\label{equ:20131120g}
\begin{split}
\liminf_{j\to+\infty}\int_{ D_v}|F_{v,t_j}|^{2}d\lambda_{n}=0.
\end{split}
\end{equation}

According to inequality \ref{equ:20131120g} and inequality \ref{equ:20131120f},
it follows that
\begin{equation}
\label{equ:13.10.11b}
\begin{split}
&\int_{ D_v}|F_{v}|^{2}d\lambda_{n}\leq0.
\end{split}
\end{equation}

As $(F_{v,t_{j}}-F,o)\in\mathcal{I}(\varphi)_{o},$
and by Remark \ref{r:131011},
we have $F_{v}\not\equiv0$,
which contradicts to inequality \ref{equ:13.10.11b}.
Then Theorem \ref{t:GZ_JM} here has thus been proved.

\subsection{Proof of Theorem \ref{t:GZ_JM201312}}
$\\$

We prove Theorem \ref{t:GZ_JM201312} by contradiction:
if Theorem \ref{t:GZ_JM201312} is not true,
then there exists $t_{j}\to +\infty$ $(j\to+\infty)$ and $B_{j}\in(0,1]$,
such that
\begin{equation}
\label{equ:20131013a}
\begin{split}
&\lim_{j\to\infty}\int_{\Delta^{n}}\frac{1}{B_{j}}\mathbb{I}_{\{-t_{j}-B_{j}<\Psi<-t_{j}\}}e^{-\psi+\log|F|^{2}}d\lambda_{n}
\\&=\lim_{j\to\infty}\int_{\Delta^{n}}\frac{1}{B_{j}}\mathbb{I}_{\{-t_{j}-B_{j}<\psi-\log|F|^{2}<-t_{j}\}}e^{-\psi+\log|F|^{2}}d\lambda_{n}=0.
\end{split}
\end{equation}

Let $t_{j}>10$, for all $j$.
Note that
\begin{equation}
\label{}
\begin{split}
&e^{-\max\{\psi-\log|F|^{2},0\}}|_{\{-t_{j}-B_{j}<\Psi<-t_{j}\}}
\\&=e^{-\max\{\psi-\log|F|^{2},0\}}|_{\{-t_{j}-B_{j}<\min\{\psi-\log|F|^{2},0\}-1<-t_{j}\}}
\\&=e^{-\max\{\psi-\log|F|^{2},0\}}|_{\{-t_{j}-B_{j}<\psi-\log|F|^{2}-1<-t_{j}\}}
\\&=e^{-\max\{\psi-\log|F|^{2},0\}}|_{\{-t_{j}-B_{j}+1<\psi-\log|F|^{2}<-t_{j}+1\}}
\\&\leq e^{0}=1
\end{split}
\end{equation}
then
\begin{equation}
\label{equ:20131204c}
\begin{split}
&\lim_{j\to\infty}\int_{\Delta^{n}}\frac{1}{B_{j}}\mathbb{I}_{\{-t_{j}-B_{j}<\Psi<-t_{j}\}}|F|^{4}e^{-\varphi-\Psi}d\lambda_{n}
\\&\leq
\lim_{j\to\infty}\int_{\Delta^{n}}\frac{1}{B_{j}}\mathbb{I}_{\{-t_{j}-B_{j}<\Psi<-t_{j}\}}e^{-2\max\{\psi-\log|F|^{2},0\}-\min\{\psi-\log|F|^{2},0\}+1}d\lambda_{n}
\\&\leq
\lim_{j\to\infty}\int_{\Delta^{n}}\frac{1}{B_{j}}\mathbb{I}_{\{-t_{j}-B_{j}<\Psi<-t_{j}\}}e^{-\max\{\psi-\log|F|^{2},0\}-\psi+\log|F|^{2}+1}d\lambda_{n}
\\&\leq\lim_{j\to\infty}\int_{\Delta^{n}}\frac{1}{B_{j}}\mathbb{I}_{\{-t_{j}-B_{j}<\Psi<-t_{j}\}}ee^{-\psi+\log|F|^{2}}d\lambda_{n}=0.
\end{split}
\end{equation}

Let $D_{v}$ be a strongly pseudoconvex domain relatively compact in $\Delta^{n}$ containing $o$,
it follows that
\begin{equation}
\label{equ:20131013}
\begin{split}
\int_{\Delta^{n}}\frac{1}{B_{j}}\mathbb{I}_{\{-t_{j}-B_{j}<\Psi<-t_{j}\}}|F|^{4}e^{-\varphi-\Psi}d\lambda_{n}
\geq \int_{D_v}\frac{1}{B_{j}}\mathbb{I}_{\{-t_{j}-B_{j}<\Psi<-t_{j}\}}|F|^{4}e^{-\varphi-\Psi}d\lambda_{n}.
\end{split}
\end{equation}

According to equality \ref{equ:20131013a} and inequality \ref{equ:20131013},
it follows that
\begin{equation}
\label{equ:20131120a}
\begin{split}\lim_{j\to+\infty}\int_{D_v}\frac{1}{B_{j}}\mathbb{I}_{\{-t_{j}-B_{j}<\Psi<-t_{j}\}}|F|^{4}e^{-\varphi-\Psi}d\lambda_{n}=0.
\end{split}
\end{equation}

By Proposition \ref{p:GZ_JM201312},
it follows that there exists $F_{v,t_{j}}$,
which is a holomorphic function on $D_{v}$
satisfying:

\begin{equation}
\label{equ:13.10.11a}
\begin{split}
&\int_{ D_v}|F_{v,t_j}-(1-b_{t_j}(\Psi))F^2|^{2}e^{-\varphi}d\lambda_{n}
\\\leq&\int_{D_v}\frac{1}{B_{j}}\mathbb{I}_{\{-t_{j}-B_{j}<\Psi<-t_{j}\}}|F|^{4}e^{-\varphi-\Psi}d\lambda_{n},
\end{split}
\end{equation}
and
$$(F_{v,t_{j}}-F^{2},o)\in\mathcal{I}(\varphi+\Psi)_{o}.$$

As $|(1-b_{t_j}(\Psi))F^{2}|$ on $D_{v}$ have a uniform bound independent of $j$,
then $\int_{ D_v}|(1-b_{t_j}(\Psi))F^{2}|^{2}e^{-\varphi}d\lambda_{n}$ have a uniform bound independent of $j$.

According to equality \ref{equ:20131120a} and inequality \ref{equ:13.10.11a},
it follows that $\int_{ D_v}|F_{v,t_j}-(1-b_{t_j}(\Psi))F^{2}|^{2}e^{-\varphi}d\lambda_{n}$ have a uniform bound independent of $j$.

Using
\begin{equation}
\label{equ:20131120b}
\begin{split}
(\int_{ D_v}|F_{v,t_j}|^{2}e^{-\varphi}d\lambda_{n})^{\frac{1}{2}}&\leq
(\int_{ D_v}|F_{v,t_j}-(1-b_{t_j}(\Psi))F^{2}|^{2}e^{-\varphi}d\lambda_{n})^{\frac{1}{2}}\\&+(\int_{ D_v}|(1-b_{t_j}(\Psi))F^{2}|^{2}e^{-\varphi}d\lambda_{n})^{\frac{1}{2}},
\end{split}
\end{equation}
we have $\int_{ D_v}|F_{v,t_j}|^{2}e^{-\varphi}d\lambda_{n}$ have a uniform bound independent of $j$.
Then there is a subsequence of $F_{v,t_j}$ denoted by $F_{v,t_j}$ without ambiguity,
which is convergent to a holomorphic function $F_{v}$ uniformly on any compact subset of $D_{v}$.
Then for any $K\subset\subset D_{v}$,
we have
\begin{equation}
\label{equ:20131120e}
\begin{split}
\int_{ K}|F_{v}|^{2}e^{-\varphi}d\lambda_{n}\leq\liminf_{j\to+\infty}\int_{ D_v}|F_{v,t_j}|^{2}e^{-\varphi}d\lambda_{n}.
\end{split}
\end{equation}
Therefore
\begin{equation}
\label{equ:20131120f}
\begin{split}
\int_{D_v}|F_{v}|^{2}e^{-\varphi}d\lambda_{n}\leq\liminf_{j\to+\infty}\int_{ D_v}|F_{v,t_j}|^{2}e^{-\varphi}d\lambda_{n}.
\end{split}
\end{equation}

As $\{(1-b_{t_j}(\varphi))F^{2}\}_{j=1,2,\cdots}$ goes to zero when $j\to+\infty$ and
$|(1-b_{t_j}(\varphi))F^{2}|$ on $D_{v}$ have a uniform bound independent of $j$,
by the Lebesgue dominated convergence theorem£¬
it follows that
\begin{equation}
\label{equ:20131120c}
\begin{split}
\lim_{j\to+\infty}\int_{ D_v}|(1-b_{t_j}(\Psi))F^{2}|^{2}e^{-\varphi}d\lambda_{n}=0.
\end{split}
\end{equation}

According to equality \ref{equ:20131120a}, inequality \ref{equ:13.10.11a},
it follows that
\begin{equation}
\label{equ:20131120d}
\begin{split}
\lim_{j\to+\infty}\int_{ D_v}|F_{v,t_j}-(1-b_{t_j}(\Psi))F^{2}|^{2}e^{-\varphi}d\lambda_{n}=0.
\end{split}
\end{equation}

Using equality \ref{equ:20131120c}, equality \ref{equ:20131120d} and inequality \ref{equ:20131120b},
we have
\begin{equation}
\label{equ:20131120g}
\begin{split}
\liminf_{j\to+\infty}\int_{ D_v}|F_{v,t_j}|^{2}e^{-\varphi}d\lambda_{n}=0.
\end{split}
\end{equation}

According to inequality \ref{equ:20131120g} and inequality \ref{equ:20131120f},
it follows that
\begin{equation}
\label{equ:13.10.11b}
\begin{split}
&\int_{ D_v}|F_{v}|^{2}e^{-\varphi}d\lambda_{n}\leq0.
\end{split}
\end{equation}

As $(F_{v,t_{j}}-F^{2},o)\in\mathcal{I}(\varphi+\Psi)_{o},$
and by Remark \ref{r:20131210} and \ref{r:131011},
we have $F_{v}\not\equiv0$,
which contradicts to inequality \ref{equ:13.10.11b}.
Then Theorem \ref{t:GZ_JM201312} here has thus been proved.

\subsection{Proof of Proposition \ref{t:GZ_open1026}}
$\\$

We prove Proposition \ref{t:GZ_open1026} by contradiction.
If $e^{-\phi}$ is integrable near $o\in\Delta^{n}$,
then there exists a strong pseudoconvex domain $\Omega\subset\subset\Delta^{n}$,
such that $e^{-\phi}$ is $L^{1}$ integrable on $\Omega$.

Without losing of generality,
we assume that $\Omega=\mathbb{B}(o,r)$,
where $r>0$ small enough.

As $e^{-\phi}$ is $L^{1}$ integrable on $\Omega$,
then
\begin{equation}
\label{equ:20131121a}\lim_{R\to+\infty}e^{R}\mu(\{\phi<-R\})=0.
\end{equation}
Therefore there exists $t_{1}>0$,
such that
\begin{equation}
\label{equ:20131121f}
\mu(\{\phi<-t_{1}+1\})<\frac{1}{6}\mu(\Omega).
\end{equation}

As $\{\phi_{m}\}_{m=1,2,\cdots}$ is convergent to $\phi$ in Lebesgue measure,
then there exists $m_{0}>0$,
such that for any $m\geq m_{0}$,
\begin{equation}
\label{equ:20131121g}
\mu(\{|\phi_{m}-\phi|\geq 1\})<\frac{1}{12}\mu(\Omega).
\end{equation}

Note that
$$(\{\phi_{m}<-t_{1}\}\setminus\{|\phi_{m}-\phi|\geq 1\})\subset\{\phi<-t_{1}+1\},$$
for any $m\geq m_{0}$.
Therefore
\begin{equation}
\label{equ:20131121b}
\mu(\{\phi_{m}<-t_{1}\})\leq\mu(\{\phi<-t_{1}+1\})+\mu(\{|\phi_{m}-\phi|\geq 1\})<\frac{1}{4}\mu(\Omega),
\end{equation}
for any $m\geq m_{0}$.

In Proposition \ref{p:GZ_JM},
let $F\equiv 1$ and $\varphi=\phi_{m}$,
then there exists a holomorphic function $F_{v,t_{0}}$ on $\Omega$,
satisfying:
\begin{equation}
\label{equ:20131121c}
F_{v,t_{0}}|_{o}=F=1
\end{equation}
and
\begin{equation}
\label{equ:20131026b}
\begin{split}
&\int_{\Omega}|F_{v,t_0}-(1-b_{t_0}(\phi_{m}))F|^{2}d\lambda_{n}
\\\leq&\int_{\Omega}(\mathbb{I}_{\{-t_{0}-1< t<-t_{0}\}}\circ\phi_{m})|F|^{2}e^{-\phi_{m}}d\lambda_{n}.
\end{split}
\end{equation}

By submean inequality of plurisubharmonic function,
it follows from \ref{equ:20131121c}
that
\begin{equation}
\label{equ:20131027a}
\begin{split}
\int_{\Omega}|F_{v,t_0}|^{2}d\lambda_{n}\geq\mu(\Omega).
\end{split}
\end{equation}

It follows from inequality \ref{equ:20131121b}, equality \ref{equ:20131121c},
and $b_{t_{0}}(t)|_{\{t\geq -t_{0}\}}=1$,
that
\begin{equation}
\label{equ:20131121d}
\begin{split}
\int_{\Omega}|(1-b_{t_0}(\phi_{m}))F|^{2}d\lambda_{n}
&=\int_{\{\phi<-t_{0}\}\cap\Omega}|(1-b_{t_0}(\phi_{m}))F|^{2}d\lambda_{n}
\\&\leq\int_{\{\phi<-t_{0}\}\cap\Omega}|F|^{2}d\lambda_{n}
=\int_{\{\phi<-t_{0}\}\cap\Omega}d\lambda_{n}
\\&\leq\mu(\{\phi<-t_{0}\})<\frac{1}{4}\mu(\Omega)
\end{split}
\end{equation}

It follows from
inequalities \ref{equ:20131121d} and \ref{equ:20131027a}
that
\begin{equation}
\label{equ:20131026a}
\begin{split}
&(\int_{\Omega}|F_{v,t_0}-(1-b_{t_0}(\phi_{m}))F|^{2}d\lambda_{n})^{1/2}
\\&\geq(\int_{\Omega}|F_{v,t_0}|^{2}d\lambda_{n})^{1/2}-
(\int_{\Omega}|(1-b_{t_0}(\phi_{m}))F|^{2}d\lambda_{n})^{1/2}
\\&\geq \mu(\Omega)^{1/2}-2^{-1}\mu(\Omega)^{1/2}=2^{-1}\mu(\Omega)^{1/2},
\end{split}
\end{equation}
for any $t_{0}>t_{1}$ and any $m\geq m_{0}$.

It follows from inequalities \ref{equ:20131026b} and \ref{equ:20131026a}
that
$$\int_{\Omega}(\mathbb{I}_{\{-t_{0}-1< t<-t_{0}\}}\circ\phi_{m})|F|^{2}e^{-\phi_{m}}d\lambda_{n}\geq 2^{-2}\mu(\Omega),$$
for any $t_{0}>t_{1}$ and any $m\geq m_{0}$.

Note that
$$\mu({\{-t_{0}-1<\phi_{m}<-t_{0}\}})e^{t_{0}+1}\geq \int_{\Omega}(\mathbb{I}_{\{-t_{0}-1< t<-t_{0}\}}\circ\phi_{m})|F|^{2}e^{-\phi_{m}}d\lambda_{n},$$
for any $t_{0}>t_{1}$ and any $m\geq m_{0}$.
Therefore
$$\mu({\{-t_{0}-1<\phi_{m}<-t_{0}\}})\geq e^{-t_{0}-1}2^{-2}\mu(\Omega),$$
for any $t_{0}>t_{1}$ and any $m\geq m_{0}$.

As $\{\phi_{m}\}_{m=1,2,\cdots}$ is convergent to $\phi$ in Lebesgue measure,
then
there exists large enough positive integer $m_{1}\geq m_{0}$,
such that
$$\mu(\{|\phi_{m_1}-\phi|\geq1\})<\frac{1}{2}e^{-t_{0}-1}2^{-2}\mu(\Omega),$$
for any $t_{0}>t_{1}$.

Note that
$$\{\phi<-t_{0}+1\}\supset(\{-t_{0}-1<\phi_{m_{1}}<-t_{0}\}\setminus\{|\phi_{m_1}-\phi|\geq1\}).$$
Then we have
\begin{equation}
\begin{split}
\mu({\{\phi<-t_{0}+1\}})&\geq\mu(\{-t_{0}-1<\phi_{m_{1}}<-t_{0}\}\setminus\{|\phi_{m_1}-\phi|\geq1\})
\\&\geq \mu(\{-t_{0}-1<\phi_{m_{1}}<-t_{0}\})-\mu(\{|\phi_{m_1}-\phi|\geq1\})
\\&\geq\frac{1}{2}e^{-t_{0}-1}2^{-2}\mu(\Omega),
\end{split}
\end{equation}
for any $t_{0}>t_{1}$,
i.e.
\begin{equation}
\label{equ:20131121e}
e^{t_{0}-1}\mu({\{\phi<-t_{0}+1\}})\geq\frac{1}{2}e^{-2}2^{-2}\mu(\Omega),
\end{equation}
for any $t_{0}>t_{1}$,
which contradicts to equality \ref{equ:20131121a}.

Proposition \ref{t:GZ_open1026} has thus been proved.

\begin{Remark}
Using the same method as in the above proof with more subtle bounds
in inequalities \ref{equ:20131121f} and \ref{equ:20131121g}, one can
obtain inequality \ref{equ:20131121e} with a lower bound
$e^{-2}\mu(\Omega)$.
\end{Remark}

\vspace{.1in} {\em Acknowledgements}. The authors would like to
thank Prof. Bo Berndtsson, Prof. J-P. Demailly, Prof. Nessim Sibony,
and Prof. Yum-Tong Siu for giving series of talks at CAS and
explaining us their related works. The authors would also like to
thank Prof. Bo Berndtsson, Prof. J-P. Demailly, Prof. Laszlo
Lempert, Prof. M. P\u{a}un, Prof. Nessim Sibony, and Prof. Gang Tian
for reading our paper \cite{GZopen-a} and giving comments. Thanks
also to Dr. Junyan Cao for explaining some results of his.

\bibliographystyle{references}
\bibliography{xbib}

\end{document}